%% file: APT_Template.tex
\documentclass{aptpub}
\numberwithin{equation}{section}
\authornames{Remco van der Hofstad, Manish Pandey} 
\shorttitle{Connectivity of random graphs after centrality-based vertex removal} 

\RequirePackage[utf8]{inputenc}
\usepackage{diagbox}
\usepackage[mathscr]{euscript}
\usepackage{enumerate}
\usepackage{bbm}
\usepackage{color}
\usepackage{todonotes}

\newcommand{\shortversion}[1]{}
\newcommand{\longversion}[1]{#1}

\newcommand{\sss}{\scriptscriptstyle}
\newcommand{\eqn}[1]{\begin{equation}#1\end{equation}}

\newcommand {\convp}{\stackrel{\sss {\mathbb P}}{\longrightarrow}}

\newcommand{\prob}{\mathbb{P}}
\newcommand{\expec}{\mathbb{E}}
\newtheorem{notation}{Notation}[section]

\newcommand{\paragraphown}[1]{\medskip
\noindent
{\bf #1}}
\newcommand{\ensymboldefinition}{$\blacktriangleleft$}
\newcommand{\qed}{\hfill \ensuremath{\Box}}

\begin{document}

\title{Connectivity of random graphs after \\centrality-based vertex removal} 

\authorone[Eindhoven University of Technology]{Remco van der Hofstad}
\authorone[Eindhoven University of Technology]{Manish Pandey}
\addressone{Department of Mathematics and Computer Science, Eindhoven University of Technology, 5600 MB Eindhoven, The Netherlands}
\emailone{r.w.v.d.hofstad@tue.nl, m.pandey@tue.nl}
\begin{abstract}
Centrality measures aim to indicate who is important in a network. Various notions of `being important' give rise to different centrality measures. In this paper, we study how important the central vertices are for the \emph{connectivity structure} of the network, by investigating how the removal of the most central vertices affects the number of connected components and the size of the giant component. We use \emph{local convergence techniques} to identify the limiting number of connected components for locally converging graphs and centrality measures that depend on the vertex's neighborhood. For the size of the giant, we prove a general upper bound. For the matching lower bound, we specialize to the case of \emph{degree centrality} on one of the most popular models in network science, the \emph{configuration model}, for which we show that removal of the highest-degree vertices destroys the giant most. 
\end{abstract}



\input{new.tex}

\end{document}

%% file: new.tex

\section{Introduction and Main Results}
\subsection{Introduction}

Complex networks are everywhere. Prominent examples include social networks, Internet, the World-Wide-Web, transportation networks, etc. In any network, it is of great interest to be able to quantify who is `important' and who is less so. This is what \emph{centrality measures} aim to do.

There are several well-known centrality measures in networks \cite[Chapter 7]{newman2018networks}, such as degree centrality, PageRank centrality and betweeness centrality. PageRank, first proposed in \cite{page1999pagerank}, can be visualised as the stationary distribution of a random walk with uniform restarts. Betweeness centrality was first defined by Anthonisse \cite{anthonisse1971rush}. For a survey on centrality measures, see Boldi and Vigna  \cite{boldi2014axioms} and Borgatti \cite{borgatti2005centrality}. 

There are many ways in which we can compare the effectiveness of centrality measures. Here, one can think of importance for spreading diseases or information \cite{wei2022identifying}, for being an information bridge between different communities, or for being an important source of information in a network of scientific papers \cite{senanayake2015pagerank}. 

In this paper, we compare centrality measures by investigating the rate of \emph{disintegration} of the network by the removal of central vertices. This approach quantifies the notion that a centrality measure is more effective when the network disintegrates more upon the removal of the most central vertices. This work is motivated by the analysis and simulations performed in \cite{mocanu2018decentralized}, where the number of connected components and the size of the giant were compared after the removal of the most central vertices based on several centrality measures from a simulation perspective.

\paragraphown{\bf Central questions.}
Our key questions are as follows:
\begin{itemize}
    \item[$\rhd$] How will the number of connected components grow with the removal of the most/least central vertices? 
    \item[$\rhd$] Does vertex removal with respect to centrality measures preserve the existence of local limits?
    \item[$\rhd$] What are the sub-critical and super-critical regimes of the giant component for vertex removals with respect to centrality measure?
    \item[$\rhd$] What is the proportion of vertices in the giant component?
\end{itemize} 

\paragraphown{\bf Main innovation of this paper.} To answer the above questions, we rely on the theory of \emph{local convergence} \cite{AldSte04, BenSch01}, see also \cite[Chapter 2]{Hofst23} for an extensive overview. We show that when considering \emph{strictly local} centrality measures, i.e., measures that depend on a fixed-radius neighborhood of a vertex, the number of connected components after the vertex removal procedure converges, and the local limit of the removed graph can be determined in terms of the original local limit. Thus, this answers the first two questions rather generally, assuming local convergence.

It is well-known that the \emph{giant} in a random graph is not determined by the local limit (even though it frequently `almost' is, see \cite{Hofs21b}). While the upper bound on the giant is always equal to the survival probability of the local limit, the matching \emph{lower bound} can be different. 
To give an example where we can prove that the giant of the vertex-removed graph equals the survival probability of its local limit, we restrict our attention to the configuration model with a given degree distribution, and degree centrality. There, we identify the giant of the vertex-removed graph, and also show that the giant is smaller when removing more degree-central vertices.

\subsection{Preliminaries}\label{def: degree centrality}

In this section, we define centrality measures and then give an informal introduction to local convergence, as these play a central role in this paper. We will assume that $|V(G)|=n$, and for convenience assume that $V(G)=\{1, \ldots, n\}\equiv[n]$. 

\subsubsection{Centrality Measures} In this section, we define \emph{centrality measures:}

\begin{defn}[Centrality measures]
For a \emph{finite} (undirected) graph $G=(V(G),E(G))$, a \emph{centrality measure} is a function $R\colon V(G) \to \mathbb{R}_{\geq0}$, where we consider $v$ to be central if $R(v)$ is large. \hfill\ensymboldefinition
\end{defn}
Centrality measures can be extended to {\em directed} graphs, but in this paper we restrict to undirected graph as in \cite{mocanu2018decentralized}. Commonly used examples of centrality measures are the following: 

\paragraphown{\textbf{Degree Centrality.}} 
In degree centrality, we rank the vertices according to their degrees and then divide the rank by the total number of vertices.
We \emph{randomly} assign ranks to vertices with the same degree, to remove ties. 
\medskip

\paragraphown{\textbf{PageRank Centrality.}}
PageRank is a popularly known algorithm for measuring centrality in the World-Wide Web \cite{BriPag98}:

\begin{defn}[PageRank centrality] Consider a finite (undirected) graph $G$. Let $e_{j,i}$ be the number of edges between $j$ and $i$. Denote the degree of vertex $i \in [n]$  by $d_{i}$. Fix a damping factor or teleportation parameter $c\in(0, 1)$. Then, PageRank is the unique probability vector $\boldsymbol{\pi}_n = (\pi_n(i))_{i\in [n]}$  that satisfies that, for every $i \in [n]$,
\begin{equation}\label{page rank equation}
       \pi_n(i)  = c\sum_{j\in [n]}\frac{e_{j,i}}{d_j}\pi_n(j) + \frac{1-c}{n}, \qquad i\in [n].
\end{equation}
Suppose $\boldsymbol{1} = (1, 1, \ldots, 1)$ is the all-one vector. Then, with $P = (p_{i,j})_{i, j \in V(G)}$ and $p_{i,j}= e_{i, j}/d_i$ the transition matrix for the random walk on $G$,
    \begin{equation}\label{page rank solution}
    \boldsymbol{\pi}_n = \left(\frac{1-c}{n}\right) \boldsymbol{1}(I-cP)^{-1}.
    \end{equation}
\hfill\ensymboldefinition
\end{defn}

\noindent

In order to use local convergence techniques \cite{GarHofLit20}, it is useful to work with the \textbf{graph-normalised} PageRank, given by
    \eqn{
    \boldsymbol{R}_n = n\boldsymbol{\pi}_n,
    \qquad\text{for which} 
    \qquad
   \frac{1}{n}\sum_{j\in [n]}R_n(j) = 1.
    }
Since $c\in(0,1)$, \eqref{page rank solution} implies
    \eqn{
    \label{power-iteration-PageRank}
    \boldsymbol{\pi}_n =\left( \frac{1-c}{n} \right) \boldsymbol{1}\sum_{k= 0}^{\infty} c^k P^k.
    }
Equation \eqref{power-iteration-PageRank} is sometimes called  \emph{power-iteration} for PageRank \cite{avrachenkov2006pagerank, bianchini2005inside, boldi2005pagerank}.
It will also be useful to consider the \textbf{\emph{finite-radius}} PageRank, 
$\boldsymbol{R}_n^{\sss(N)}$, for $N\in \mathbb{N}$, defined by
    \begin{equation}
    \label{finite radius pagerank}
    \boldsymbol{R}_n^{\sss(N)} :=\left( \frac{1-c}{n}\right) \boldsymbol{1}\sum_{k= 0}^{N} c^k P^k,
    \end{equation}
which approximates PageRank by a finite number of powers in \eqref{power-iteration-PageRank}. PageRank has been intensely studied on various random graph models \cite{MR4474527, MR3297352, MR3683363, GarHofLit20,  MR2796677, MR4074702,  MR2522875}, with the main focus being to prove or disprove that PageRank has the same power-law exponent as the in-degree distribution.

\paragraphown{\textbf{Other popular centrality measures.}} \emph{Closeness centrality} measures the average distance of a vertex from a randomly chosen vertex. The higher the average, the lower is the centrality index, and vice-versa. Interestingly, \cite{Evans2022} predicts that closeness centrality is closely related to degree centrality, at least for locally tree-like graphs. \emph{Betweeness centrality} measures the extent to which vertices are important to \emph{interconnect} different vertices.
\medskip

In this paper, we mainly work with \emph{strictly-local} centrality measures:

\begin{defn}[Strictly local centrality measures]
For a \emph{finite} (undirected) graph $G=(V(G),E(G))$, a \emph{strictly local centrality measure} is a centrality measure $R \colon V(G) \to \mathbb{R}_{\geq0}$, such that there exists an $r\in \mathbb{N}$ and for each vertex $v\in V(G)$, $R(v)$ depends only on the graph through the neighborhood 
    \eqn{
    B_{r}^{\sss(G)}(v) := \{ u \in V(G)\colon {\rm dist}_{\sss G}(u,v) \leq r\},
    }
where ${\rm dist}_{\sss G}(u,v)$ denotes the graph distance between $u$ and $v$ in $G$.
\hfill\ensymboldefinition
\end{defn}

PageRank is very well approximated by its strictly local version as in \eqref{finite radius pagerank} \cite{avrachenkov2007monte,boldi2014axioms, GarHofLit20}.

\subsubsection{Local convergence of undirected random graphs}
In this section, we informally introduce local convergence of random graphs, which describes what a graph locally looks like from the perspective of a uniformly chosen vertex, as the number of vertices in a graph goes to infinity. 
For example, the sparse Erd\H{o}s-R\'enyi random graph, which is formed by bond-percolation on the complete graph, locally looks like a Poisson branching process, as $n$ tends to infinity \cite[Chapter 2]{Hofst23}. Before moving further, we discuss the following notations:
\begin{notation}[Probability convergence]Suppose $(X_n)_{n\geq 1}$, $(Y_n)_{n\geq 1}$ are two sequences of random variables and $X$ is a random variable.
    \begin{enumerate}
        \item We write $X_n \overset{\sss{\prob}/d}{\to} X$ when $X_n$ converges in probability/distribution to $X$.
        \item For two sequences $(f(n))_{n\geq 1}$ and $(g(n))_{n \geq 1}$, we write $f(n) = o(g(n))$ when $\lim_{n\to \infty}f(n)/g(n) = 0.$
        \item We write $X_n = o_{\sss\mathbb{P}}(Y_n)$ when $X_n/Y_n \overset{\sss\prob}{\to} 0$.
    \end{enumerate}
\end{notation}
Now let us give the informal definition of local convergence. We will rely on two types of local convergence, namely, local weak convergence and local convergence in probability. 
Let $o_n$ denote a uniformly chosen vertex from $V(G_n)$. Local weak convergence means that
\begin{equation}\label{eq: local weak convergence} 
    \prob\left({B_r^{\sss(G_n)}(o_n)\cong (H, o')}\right) \to \bar{\mu}(B_r^{\sss(\bar{G})}(o)\cong (H,o')),
\end{equation}
for all rooted graphs $(H, o')$, where a rooted graph is a graph with a distinguished vertex in the vertex set $V(H)$ of $H$. In \eqref{eq: local weak convergence}, $(\bar{G}, o) \sim \bar{\mu}$ is a {\em random} rooted graph, which is called the local weak limit. For local convergence in probability, instead, we require that
\begin{equation}\label{eq: local convergence in probability}
    \frac{1}{|V(G_n)|}\sum_{v\in V(G_n)}\mathbbm{1}_{\left\{B_r^{\sss(G_n)}(v)\cong (H, o')\right\}} \overset{\sss\prob}{\to} \mu(B_r^{\sss(G)}(o)\cong (H,o')),
\end{equation}
holds for all rooted graphs $(H, o')$. In \eqref{eq: local convergence in probability}, $(G, o) \sim \mu$ is a random rooted graph which is called the local limit in probability (and bear in mind that $\mu$ can possibly be a random measure on rooted graphs). Both \eqref{eq: local weak convergence} and \eqref{eq: local convergence in probability} describe the convergence of the proportions of vertices around which the graph locally looks like a certain specific graph. We discuss this definition more formally in Section \ref{sec-LC-RG}. We now turn to our main results.
\subsection{Main results}
\label{sec-main-results}
In this section, we state our main results. In Section \ref{sec-vertex-removal-local}, we discuss our results that hold for general strictly local centrality measures on locally converging random graph sequences. In Section \ref{sec-vertex-removal-giant}, we investigate the size of the giant after vertex removal. Due to the non-local nature of the giant, there we restrict to \emph{degree centrality} on the \emph{configuration model}.

\subsubsection{Strictly local centrality measures}
\label{sec-vertex-removal-local}
We first define our vertex removal procedure based on centrality:
\begin{defn}[Vertex removal based on centrality threshold]
Let $G$ be an arbitrary graph. Define $G(R, r)$ to be the graph obtained after removing all the vertices $v$ for which $R(v)> r$. We call this the \textbf{$\boldsymbol{r}$-killed graph of $G$.}
\hfill\ensymboldefinition
\end{defn} 

\begin{thm}[Continuity of vertex removal]
\label{theorem-LC-vertex-removal} Let $R$ be a strictly local centrality measure and ${(G_n)}_{n\geq 1}$ a sequence of random graphs that converges locally weakly/ locally in probability to $(G, o) \sim \mu$. Then, $(G_n(R, r))_{n\geq 1}$ converges locally weakly/ locally in probability to $(G(R, r), o)$, respectively.  
\end{thm}

Theorem \ref{theorem-LC-vertex-removal} means that vertex removal with respect to a strictly local centrality threshold is \emph{continuous} with respect to the local convergence.

Let $C_r(o)$ denote the connected component containing the root in the $r$-killed graph, $(G(R,r), o)\sim \mu$. As a corollary to Theorem \ref{theorem-LC-vertex-removal}, we bound the limiting value of the proportion of vertices in the giant component in $G_n(R, r)$ in probability by the survival probability of $(G(R, r), o)\sim \mu$, which is $\mu(|C_r(o)| = \infty)$.
\begin{cor}[Upper bound on giant]\label{upperbound for proportion}
       Denote $v(C_1(G_n(R, r))$ to the number of vertices in the giant component of $G_n(R, r)$. Under the conditions of Theorem \ref{theorem-LC-vertex-removal}
        $$\lim_{n\to \infty}\mathbb{P}(v(C_1(G_n(R, r)))\leq n(\zeta + \varepsilon)) = 1, $$
  for all $\varepsilon>0$, where $\zeta = \mu(|C_r(o)| = \infty).$      
\end{cor}

 Let $K_n^r(R)$ denote the number of connected component in the killed graph $G_n(R, r)$. As another corollary to Theorem \ref{theorem-LC-vertex-removal}, we give the convergence properties for $K_n^r(R)$:

\begin{cor}[Number of connected components]\label{connected component general} Under the conditions of Theorem \ref{theorem-LC-vertex-removal}
\begin{enumerate}[(a)]
    \item If $G_n$ converges locally in probability to $(G, o) \sim \mu$, then
    \eqn{\frac{K_n^r(R)}{n} \overset{\sss\prob}{\to} \mathbb{E}_{\mu}\left[\frac{1}{|C_r(o)|}\right].
    }
    \item If ${(G_n)}_{n\geq 1}$ converges locally weakly to $(\bar{G}, \bar{o}) \sim \bar{\mu}$, then
    \eqn{\frac{\expec[K_n^r(R)]}{n} \to \mathbb{E}_{\bar{\mu}}\left[\frac{1}{|C_r(o)|}\right].
    }
\end{enumerate}
\end{cor}

\subsubsection{Degree centrality and configuration model}
\label{sec-vertex-removal-giant}
In this section, we restrict to degree centrality in the configuration model. Before starting with our main results, we introduce the configuration model:

\paragraphown{\bf Configuration Model.} 
The configuration model was introduced by Bollob\'as \cite{Boll80b}, see also \cite[Chapter 7]{Hofs17} and the references therein for an extensive introduction. It is one of the simplest possible models for generating a random graph with a given degree distribution. Written as ${\rm CM}_n(\boldsymbol{d})$, it is a random graph on $n$ vertices having a given degree sequence $\boldsymbol{d}$, where $\boldsymbol{d} = (d_1, d_2,\ldots, d_n) \in\mathbb{N}^{n}.$ The giant in the configuration model has attracted considerable attention in e.g.,  \cite{MR3343756, JanLuc07, molloy1995critical, MR1664335, Hofs21b}. It is also known how the tail of the limiting degree distribution influences the size of giant \cite{deijfen2018tail}. Further, the diameter and distances in the supercritical regime have been studied in \cite{van2005distances, van2007distances, hofstad2007phase}, while criteria for the graph to be simple appear in \cite{MR2266448, janson2009probability, MR3317354}. In this paper we assume that the degrees satisfy the following usual conditions:
\begin{cond}[Degree conditions]
\label{cond-degrees}
Let $\boldsymbol{d} = (d_1, d_2, \ldots , d_n)$ denote a degree sequence. Let $n_j=\{v\colon d_v=j\}$ denote the number of vertices with degree $j$. We assume that there exists a probability distribution $(p_j)_{j\geq 1}$ such that the following hold:
\begin{itemize}
\item[(a)] $\lim_{n\rightarrow \infty}n_j/n= p_j$;
\item[(b)] $\lim_{n\rightarrow \infty}\sum_{j\geq 1} j n_j/n= \sum_{j\geq 1} jp_j<\infty$.
\end{itemize}
\end{cond}
Let $D$ be a non-negative integer random variable with probability mass function $(p_j)_{j\geq 1}$. These regularity conditions ensures that the sequence of graphs converges locally in probability to a unimodular branching process, with degree distribution $D$. See \cite[Chapter 4]{Hofst23} for more details. 

\paragraphown{\textbf{Generalized vertex removal based on degrees.}}
We wish to study the effect of the removal of the $\alpha$ proportion of vertices with the highest or lowest degrees on the giant of the configuration model. 
For this, we define $\alpha$-sequences with respect to a probability mass function as follows:  
\begin{defn}[$\alpha$-sequence]
Fix $\alpha \in (0,1)$. Let $\boldsymbol{r}:= (r_i)_{i\geq1}$ is a sequence of elements of $[0,1]$ satisfying
\begin{align} \label{eq1}
      \mathbb{E}[r_D&]=\sum_{i\geq1}p_ir_i = \alpha.
\end{align}
Then, $\boldsymbol{r}$ is called an $\alpha$-sequence with respect to $\boldsymbol{p} = (p_j)_{j\geq 1}$.
\end{defn}
Suppose $(G_n)_{n\geq 1}$ is a sequence of random graph converging locally in probability to $(G,o)\sim\mu$ and the limiting degree distribution be $D$, with probability mass function as $\boldsymbol{p} = (p_j)_{j\geq 1}$. Suppose $\boldsymbol{r}=(r_j)_{j\geq 1}$ is an $\alpha$-sequence with respect to $\boldsymbol{p}$, then we define vertex removal based on $\alpha$-sequences as follows:
\begin{defn}[Vertex removal based on $\alpha$-sequences]
    Remove $\lfloor n r_i p_i\rfloor$ vertices of degree $i$ from $G_n$ uniformly at random, for each $n\geq 1$. This gives us the {\em vertex  removed graph according to the $\alpha$-sequence $\boldsymbol{r}=(r_j)_{j\geq 1}$, denoted by $(G_{n,\boldsymbol{r}})_{n\geq 1}$.}\hfill\ensymboldefinition
\end{defn}

\begin{rem}
        In $G_{n, \boldsymbol{r}}$, we asymptotically remove an $\alpha$ proportion of vertices because, due to Condition \ref{cond-degrees} and the dominated convergence theorem,
        \begin{equation}
            \lim_{n\to \infty}\sum_{j\geq 1}r_j \frac{n_j}{n} = \ \sum_{j\geq 1}r_j p_j = \alpha.
        \end{equation}
\end{rem}
\paragraphown{\textbf{Results for the Configuration Model.}}
Let $(G_n)_{n\geq1}$ be a sequence of random graphs satisfying $G_n \sim {\rm CM}_n(\boldsymbol{d})$. 
Throughout the paper, we shall assume that 
$$\nu := \frac{\mathbb{E}[D(D-1)]}{\expec[D]}>1.$$ 
Indeed for $\nu < 1$, we already know that there is no giant to start with, and it cannot appear by removing vertices.
\begin{defn}[$\boldsymbol{r}$-set]\label{def-r-set}
    Suppose $\boldsymbol{r} = (r_j)_{j\geq 1}$ is an $\alpha$-sequence with respect to $\boldsymbol{p} = (p_j)_{j\geq 1}$. Let $X = [0, 1]^{\mathbb{N}}$ be the set of all sequences in $[0,1]$. Define the set $S(\boldsymbol{r}) = \{(\boldsymbol{r^{\sss(n)}})_{n\geq 1} \in X^{\mathbb{N}},\text{ s.t. }  \lim_{n\to \infty}r_i^{\sss(n)} = r_i \quad \forall i\geq 1\}.$
    \hfill\ensymboldefinition
\end{defn}

An $\boldsymbol{r}-$set $S(\boldsymbol{r})$ can be thought of as the set of those sequences in $[0, 1]^{\mathbb{N}}$ which converge to $\boldsymbol{r}$ component wise. Let $v(C_1({\boldsymbol{r}}))$ and $e(C_1({\boldsymbol{r}}))$ denote the number of vertices and edges in the giant component of $G_{n,{\boldsymbol{r}}}$. The following theorem describes the giant in $G_{n,{\boldsymbol{r^{\sss(n)}}}}$:

\begin{thm}[Existence of giant after vertex removal]\label{giant component condition and the proportion of giant}
    Let $(\boldsymbol{r}^{\sss(n)})_{n\geq 1}$ be a sequence from the $\boldsymbol{r}$-set $S(\boldsymbol{r})$. Then the graph $G_{n,\boldsymbol{r}^{\sss(n)}}$ has a giant component if and only if $\nu_{\boldsymbol{r}}>1$, where 
\begin{equation}\label{condition for giant}
 \nu_{\boldsymbol{r}} := \frac{\mathbb{E}[D(D-1)(1-r_D)]}{\mathbb{E}[D]}.
\end{equation}
For $\nu_{\boldsymbol{r}}>1$, 
\begin{equation}\label{rho}
     \frac{v(C_1(\boldsymbol{r}^{\sss(n)}))}{n}\overset{\sss\prob}{\to}\rho(\boldsymbol{r}) := 1- \alpha -2\mathbb{E}[Dr_D]\eta_{\boldsymbol{r}}-\sum_{i\geq 1}(1-r_i)p_i{\eta^i_{\boldsymbol{r}}}.
\end{equation}
\begin{equation}\label{rho1}
    \frac{e(C_1(\boldsymbol{r}^{\sss(n)}))}{n}\overset{\sss\prob}{\to}e(\boldsymbol{r}):=\frac{\expec[D]}{2}(1-{\eta^{2}}_{\boldsymbol{r}})-\expec[D r_D](1-{\eta_{\boldsymbol{r}}}) .
\end{equation}
In the above equations, ${\eta_{\boldsymbol{r}}} \in (0, 1]$ is the smallest solution of
\begin{equation}\label{eta satisfies this}
    g_{\boldsymbol{r}}'({\eta_{\boldsymbol{r}}}) = \frac{\expec[D]}{\beta_{\boldsymbol{r}}} {\eta_{\boldsymbol{r}}},
\end{equation}
where $\beta_{\boldsymbol{r}} = \mathbb{E}[Dr_D]+1-\alpha$ and $g_{\boldsymbol{r}}(\cdot)$ is the generating function for a random variable dependent on $\boldsymbol{r}$ given by
\begin{equation}\label{g_r}
    g_{\boldsymbol{r}}(s) =\frac{\sum_{i=1}^{\infty}(1-r_i)p_is^i + \mathbb{E}[Dr_D]s}{\beta_{\boldsymbol{r}}}.
\end{equation}
\end{thm}

\begin{rem}[Class of sequences with the same limiting proportion]
    For the class of sequences which converge to a given $\alpha$-sequence component wise, we always have the same limiting proportion of vertices/edges in the giant component. Thus, choosing our sequence $(\boldsymbol{r}^{\sss(n)})_{n\geq 1}$ appropriately, it is always possible to remove an exact $\alpha$ proportion of the vertices from each $G_n$ having the same limiting graph. 
\hfill\ensymboldefinition
\end{rem}

We next investigate the effect of removing the $\alpha$ proportion of vertices with the highest/lowest degrees. For that we first define quantiles:
\begin{defn}[Degree centrality quantiles]
\label{def-vertex-rem-degree}
For each $\alpha \in (0, 1)$, let $k_\alpha$, the top $\alpha$-quantile, satisfy 
    \begin{equation}
    \label{eq: k alpha}
    \mathbb{P}(D> k_\alpha) < \alpha \quad \text{and}\quad \mathbb{P}(D\geq k_\alpha) \geq \alpha.
\end{equation}
\noindent
Similarly for each $\alpha \in (0, 1)$, let $l_\alpha$, the bottom $\alpha$-quantile, satisfy 
    \begin{equation}
    \label{eq: l alpha}
    \mathbb{P}(D< l_\alpha) < \alpha\quad \text{and}\quad\mathbb{P}(D\leq l_\alpha) \geq \alpha.
    \end{equation}   
\hfill\ensymboldefinition
\end{defn}
\begin{defn}[$\alpha$-sequences corresponding to top and bottom removal]\label{definition: upper and lower removal alpha sequence}
    Let $k$ be the top $\alpha$-quantile for the degree distribution $D$. Define $\Bar{\boldsymbol{r}}(\alpha)$ to have coordinates equal to zero until the $k$th coordinate, which is $(\alpha-\mathbb{P}(D>k))/p_k$, and ones thereafter. 
    Then $\Bar{\boldsymbol{r}}(\alpha)$ is the \textbf{$\alpha$-sequence corresponding to the top $\alpha$-removal}.
    
    Similarly, let $l$ be the lower $\alpha$-quantile for the degree distribution $D$. Define $\underline{\boldsymbol{r}}(\alpha)$ to have  coordinates equal to one until the $l$th coordinate, which is $(\alpha-\mathbb{P}(D<l))/p_l$, and zeroes thereafter. 
    Then $\underline{\boldsymbol{r}}(\alpha)$ is the \textbf{$\alpha$-sequence corresponding to the bottom $\alpha$-removal}.
\hfill\ensymboldefinition
\end{defn}


\begin{cor}[Highest/lowest $\alpha$-proportion removal]\label{alpha critical and size of giant k alpha} 
Let $\bar{\alpha}_c$ and $\underline{\alpha}_c$ be defined as
\begin{align*}
    &\bar{\alpha}_c = \inf_{}\{\alpha>0:\mathbb{E}[D(D-1)\mathbbm{1}\{D\leq k_\alpha\}] > \mathbb{E}[D] + k_\alpha(k_\alpha - 1)\left(\alpha - \mathbb{P}(D>k_\alpha)\right)\},\\
    &\underline{\alpha}_c = \inf_{}\{\alpha>0:\mathbb{E}[D(D-1)\mathbbm{1}\{D\ \geq l_\alpha\}] > \mathbb{E}[D]  +  l_\alpha(l_\alpha - 1)\left(\alpha - \mathbb{P}(D< l_\alpha)\right)\}.
\end{align*}
\begin{enumerate}
    \item[(1)] Let $v(\bar{C}_1^\alpha)$ and $e(\bar{C}_1^\alpha)$ be the number of vertices and edges in the largest connected component of the top $\alpha$-proportion degree vertex removed graph, respectively. If $\alpha \geq \bar{\alpha}_c$, then there is no giant component, i.e., $v(\bar{C}_1^{\alpha})/n\convp 0$. If $\alpha < \bar{\alpha}_c$, then the giant exists and 
    \begin{align}
    \frac{v(\bar{C}_1^{\alpha})}{n}\overset{\sss\prob}{\to} \rho(\bar{\boldsymbol{r}}(\alpha))>0\qquad \text{ and } \qquad\frac{e(\bar{C}_1^{\alpha})}{n}\overset{\sss\prob}{\to} e(\bar{\boldsymbol{r}}(\alpha))>0.
    \end{align}

    
                        
                         
                        
                        

\item[(2)] Let $v(\underline{C}_1^\alpha)$ and $e(\underline{C}_1^\alpha)$ be the number of vertices and edges in the largest connected component of the lowest $\alpha$ proportion degree vertex removed graph, respectively. If $\alpha \geq \underline{\alpha}_c$, then there is no giant, i.e., $v(\underline{C}_1^{\alpha})/n\convp 0$. If $\alpha < \underline{\alpha}_c$, then the giant exists and
\begin{align}
    \frac{v(\underline{C}_1^{\alpha})}{n}\overset{\sss\prob}{\to} \rho(\underline{\boldsymbol{r}}(\alpha))>0\qquad \text{ and } \qquad\frac{e(\underline{C}_1^{\alpha})}{n}\overset{\sss\prob}{\to} e(\underline{\boldsymbol{r}}(\alpha))>0.
\end{align}



                        
                         
                        
                        

\end{enumerate}

\end{cor}



We next give bounds on the size of the giant in the vertex removed graph:
\begin{thm}[Bounds for proportion of giant]\label{bound for rho} Let ${\boldsymbol{r}}$ be an $\alpha$-sequence. Then
   \begin{equation}\label{rho bounds}
       \mathbb{E}[D({\eta_{\boldsymbol{r}}}-r_D)](1-{\eta_{\boldsymbol{r}}})\leq\rho({\boldsymbol{r}}) \leq\mathbb{E}[D(1-r_D)](1-{\eta_{\boldsymbol{r}}}).
   \end{equation}
In the above equation, $\eta_{\boldsymbol{r}} \in (0, 1]$ satisfies \eqref{eta satisfies this}. Furthermore,
\begin{align}\label{upper bound for e independent of eta}
  \rho({\boldsymbol{r}}) \leq \frac{\mathbb{E}[D(1-r_D)]^2}{\expec[D]}\qquad\text{ and } \qquad e({\boldsymbol{r}}) \leq \frac{\mathbb{E}[D(1-r_D)]^2}{2\expec[D]},
\end{align}
and
\begin{align}\label{upper bound for rho independent of eta}
    \rho({\boldsymbol{r}})\leq 1- \alpha -\frac{2\mathbb{E}[Dr_D]^2}{\expec[D]} < 1- \alpha -\frac{2\alpha^2}{\expec[D]}. 
\end{align}
Additionally, for $\alpha$-sequences which are positively correlated with degree, we have
\begin{equation}\label{eq: inequality improved for positively correlated alpha sequences}
     \rho({\boldsymbol{r}})\leq 1 -\alpha - 2\alpha^2\expec[D].
\end{equation}
\end{thm}
\begin{rem}
 $\bar{\boldsymbol{r}}(\alpha)$ is positively correlated with degree. Thus, for top $\alpha$-proportion removal \eqref{eq: inequality improved for positively correlated alpha sequences} holds
 \begin{equation}
     \rho({\bar{\boldsymbol{r}}}(\alpha))\leq 1 -\alpha - 2\alpha^2\expec[D].
 \end{equation}
\end{rem}

The next logical question is which $\alpha$-sequence destroys the graph the most. Intuitively, the size of the giant should be the smallest if we remove the top $\alpha$-proportion degree vertices and the largest when we remove the bottom $\alpha$ proportion degree vertices, i.e., $\rho(\bar{\boldsymbol{r}}(\alpha)) \leq \rho(\boldsymbol{r}) \leq \rho(\underline{\boldsymbol{r}}(\alpha))$ and $e(\bar{\boldsymbol{r}}(\alpha)) \leq e(\boldsymbol{r}) \leq e(\underline{\boldsymbol{r}}(\alpha))$. To prove this, we first define a partial order over the set of all $\alpha$-sequences that is able to capture which vertex-removal is more effective for the destruction of the giant in the configuration model. For this, we recall that, for two non-negative measures $\mu_1$ and $\mu_2$ on the real line, we say that $\mu_1$ is stochastically dominated by $\mu_2$, and write $\mu_1\preccurlyeq_{\rm{st}} \mu_2$, if $\mu_1\left([K, \infty)\right) \leq \mu_2\left([K, \infty)\right)$ for all $K \in \mathbb{R}$. We next extend this to a partial order over the set of all $\alpha$-sequences:
\begin{defn}[Stochastically dominant sequences]
    Let ${\boldsymbol{r}}=(r_j)_{j\geq 1}$ and ${\boldsymbol{r'}}=(r'_j)_{j\geq 1}$ be $\alpha_1$- and $\alpha_2$-sequences,           respectively. Then the sequences ${\boldsymbol{q}} = (p_j r_j)_{j\geq 1}$ and ${\boldsymbol{q'}} = (p_j r'_j)_{j\geq 1}$ form finite non-negative measures. We say that $\boldsymbol{r'}$ stochastically dominates $\boldsymbol{r}$, or $\boldsymbol{r} \preccurlyeq_{\boldsymbol{p}} \boldsymbol{r'}$, if $\boldsymbol{q} \preccurlyeq_{\rm{st}} \boldsymbol{q'}$.\hfill\ensymboldefinition
\end{defn}
The following theorem answers which is the best and the worst way to remove an $\alpha$-proportion vertices from a configuration model and compares other $\alpha$-sequences:
\begin{thm}[Comparison between $\alpha$-sequences]
\label{theorem-comparison-rs}
Let ${\boldsymbol{r}}=(r_j)_{j\geq 1}$ and ${\boldsymbol{r}}'=(r'_j)_{j\geq 1}$ be two $\alpha$-sequences such that ${\boldsymbol{r}}' \preccurlyeq_{\boldsymbol{p}} {\boldsymbol{r}}$. Then,
\begin{align*}
     \rho(\bar{\boldsymbol{r}}(\alpha)) \leq \rho(\boldsymbol{r})\leq \rho(\boldsymbol{r}') \leq \rho(\underline{\boldsymbol{r}}(\alpha)) \quad \text{and} \quad e(\bar{\boldsymbol{r}}(\alpha)) \leq e(\boldsymbol{r})\leq e(\boldsymbol{r}') \leq e(\underline{\boldsymbol{r}}(\alpha)).
\end{align*}
As a result, the critical $\alpha$'s for the bottom and top vertex removal satisfy $\underline{\alpha}_c > \Bar{\alpha}_c.$
\end{thm}

The following corollary is stronger than the theorem (as the corollary immediately implies the theorem), but it follows from the proof of Theorem \ref{theorem-comparison-rs}:
\begin{cor}[General comparison]\label{general comparison}
   Suppose $\boldsymbol{r},\boldsymbol{r}' \in [0, 1]^\mathbb{N}$ satisfy $\boldsymbol{r} \preccurlyeq_{\boldsymbol{p}} \boldsymbol{r}'$. Then $\rho(\boldsymbol{r}')\leq \rho(\boldsymbol{r})$ and $e(\boldsymbol{r}') \leq e(\boldsymbol{r})$. 
\end{cor}


\subsection{Overview of the proof}
In this section, we give an overview to the proofs. The theorems in this paper are of two types: The first concern arbitrary strictly local centrality measures on a sequence of locally converging random graphs, whereas the second concerns degree centrality on the configuration model. We discuss these separately in the next two sections.
\subsubsection{Strictly local centrality measures}
In this section, we discuss the proof structure for Theorem \ref{theorem-LC-vertex-removal}, which involves the following steps: 
\begin{enumerate}
    \item Proof that the vertex removal function with respect to a strictly local centrality measure is \emph{continuous} with respect to the rooted graph topology. As a result, if a function on the set of rooted graphs is continuous, then this function composed with the function of vertex removal with respect to the strictly local centrality $R$ is also continuous.
    \item Next we use the above observations and the local convergence of the sequence $(G_n)_{n\geq 1}$ to complete the proof of Theorem \ref{theorem-LC-vertex-removal}, and Corollaries \ref{upperbound for proportion} and \ref{connected component general} are immediate consequences.
\end{enumerate}

\subsubsection{Configuration model and degree centrality}\label{Intermediate construction}  
In this section, we discuss a construction, motivated by \cite{janson2009percolation}, and how it helps in the study of vertex removal procedure on the configuration model. 

Let $G = (V,E)$ be a graph. Suppose $V' \subset V$ is the set of vertices that we wish to remove. We will start by constructing an intermediate graph. which we call the {vertex-\em exploded graph}, and denote it by $\tilde{G}(V')$. To obtain $\tilde{G}(V')$ from $G$, we follow the following steps:
\begin{enumerate}
\item Index the vertices of the set $V'= \{v_1, v_2, \ldots, v_m\}$, where $m=|V'|$.
 \item \emph{Explode} the first vertex from the set $V'$ into $d_{v_1}$ vertices of degree one.
 \item Label these new degree-one vertices {\em red}.
 \item Do this for each vertex of $V'$.
\end{enumerate}
We denote the number of vertices and degree sequence of $\tilde{G}(V')$ by $\tilde{n}$ and its degree sequence by $\tilde{\boldsymbol{d}}$. The removal of red vertices, together with the (single) edge adjacent to them, from $\tilde{G}(V')$ gives the $V'$ removed graph. Next, we show that the vertex-exploded graph remains a configuration model, if vertices are removed with respect to their degree sequence. This means that we only look at the degree sequence to decide which vertices to remove, independently of the graph. It is this property that makes the configuration model amenable for vertex removal based on degree centrality:
 
\begin{lem}[Exploded configuration model is still a configuration model]\label{still a configuration model after vertex explosion with respect to degree}Let $G_n = {\rm CM}_n(\boldsymbol{d})$ be a configuration model with degree sequence $\boldsymbol{d} = (d_i)_{i\in [n]}$. Let $\tilde{G}_n$ be the vertex-exploded graph formed from $G_n$, where vertices are exploded with respect to the degree sequence. Then, $\tilde{G}_n$ has the same distribution as ${\rm CM}_{\tilde{n}}(\tilde{\boldsymbol{d}}).$
\end{lem}

\begin{proof}
A degree sequence in a configuration model can be viewed as half-edges incident to vertices. These half-edges are matched via a uniform perfect matching. Suppose one explodes a vertex of degree $i$ in ${\rm CM}_n(\boldsymbol{d})$, then we get a new degree distribution. When one explodes a vertex of degree $i$, we still have these $i$ half-edges, but these half-edges will now be incident to a unique newly formed (red) vertex. Since we are just changing the vertex to which these half-edges are incident, the half-edges in the vertex-exploded graph are still matched via a uniform perfect matching, so that the newly formed graph is also a configuration model, but now with $\tilde{n}$ vertices and degree sequence $\Tilde{\boldsymbol{d}}$. 
\end{proof}

This lemma is useful because of the following result by Janson and Luczak \cite{JanLuc07} to study the giant component in a configuration model:

\begin{thm}[Janson and Luczak \cite{JanLuc07}]\label{lwc}
Consider ${\rm CM}_n(\boldsymbol{d})$ satisfying regularity Conditions \ref{cond-degrees} with limiting degree distribution $D$. Let $g_{D}(x) := \sum_{j} p_j x^j$ be the probability generating function of $D$. Assume $p_1>0$. 
\begin{enumerate}[(i)]
    \item If $\mathbb{E}[D(D-2)]>0$, then there exists unique $\eta \in (0, 1)$ such that $g_D^{'}(\eta) = \expec[D] \eta$ and
    \begin{enumerate}[(1)]
        \item $\frac{v(C_1)}{n} \overset{\sss\prob}{\to} 1 - g_D(\eta)$;
        \item $\frac{v_j(C_1)}{n} \overset{\sss\prob}{\to} p_j(1 - \eta ^{j})$ for all $j\geq 0 $;
        \item $\frac{e(C_1)}{n} \overset{\sss\prob}{\to} \frac{\expec[D]}{2}(1-\eta^2)$.
    \end{enumerate}
    \item If $\mathbb{E}[D(D-2)] \leq 0$, then $\frac{v(C_1)}{n}\overset{\sss\prob}{\to}0$ and $\frac{e(C_1)}{n} \overset{\sss\prob}{\to} 0 $,
\end{enumerate}
where $C_1$ denotes the largest component of ${\rm CM}_n(\boldsymbol{d})$.
\end{thm}

Using this theorem on the vertex-exploded graph, we obtain the size of giant (vertices and edges). Removal of the red vertices in the vertex-exploded graph gives the required vertex removed graph. This fact is used to complete the proof of Theorem \ref{giant component condition and the proportion of giant}. The remaining results are proved by carefully investigating the limiting proportion of vertices and edges in the giant, and how they relate to the precise $\alpha$-sequence.

\subsection{Discussion and open problems}
In this section, we discuss the motivation of the problem and some open problems. The problem of vertex removal with respect to the centrality measures was motivated by \cite{mocanu2018decentralized}, which showed by simulation the size of the giant and the number of connected components behave for vertex removal based on different centrality measures. These plots were used to compare the effectiveness of these centrality measures.

A related problem is whether we can compare the size of giants in configuration models if their limiting degree distributions have some stochastic ordering. This question can be answered using the notion of $\varepsilon$-transformation as discussed in Definition \ref{def-varepsilon-transformation} below. Let $\rho_{\sss\rm{CM}}(\boldsymbol{p})$ denote the size of giant for the configuration model case when the limiting degree distribution is $\boldsymbol{p}$:
\begin{thm}[Stochastic ordering and giant in configuration model]\label{Th: stochastic ordering in congiguration model}
    If $\boldsymbol{p}\preccurlyeq_{\rm{st}} \boldsymbol{q},$ then $\rho_{\sss\rm{CM}}(\boldsymbol{p}) \leq \rho_{\sss\rm{CM}}(\boldsymbol{q})$.
\end{thm}
A similar result is proved in \cite{leskela2015impact} for increasing concave ordering (a stochastic order different from the one used in this paper), but it requires a strong additional assumption that we can remove.

\paragraphown{\bf Open problems.} We close this section with some related open problems:


\begin{enumerate}
    \item What is the size of the configuration model giant for the vertex removed graph in case of strictly local centrality measures? Which strictly local centrality measure performs the best? Can we extend this to non-local centrality measures?
    \item Strictly local approximations to non-local centrality measures: PageRank can be very well approximated by the strictly local approximation in \eqref{finite radius pagerank}. Can betweeness, closeness, and other important centrality measures also be well approximated by strictly local centrality measures? What graph properties are needed for such an approximation to be valid? Is it true on configuration models, as predicted for closeness centrality in \cite{Evans2022} and suggested for both closeness and betweeness in \cite{hinne2011local}? If so, then one can use these strictly local centrality measures for vertex removal procedures.
    \item Stochastic monotonicity for the number of connected components: If one starts removing vertices from a graph, the number of connected components first increases and then decreases. Where does it start decreasing? 
    \item Extension of results to directed graphs: can we extend the results of this paper to directed graphs? For the directed graphs, there are several connectivity notions, while vertices have {\em two} degrees, making degree  centrality less obvious. 
\end{enumerate}

\section{Proofs: strictly local centrality measures}
\label{sec-proofs-LC}
 
In this section, we define \emph{local convergence} of random graphs and then prove the results stated in Section \ref{sec-vertex-removal-local}. 

\subsection{Local convergence of undirected random graphs}
\label{sec-LC-RG}

To define local convergence, we introduce isomorphisms of rooted graphs, a metric space on them, and finally convergence in it. For more details, refer to \cite[Chapter 2]{Hofst23}

\begin{defn}[Rooted (multi)graph, rooted isomorphism, and neighborhood]~
\begin{enumerate}
    \item We call a pair $(G,o)$ a {\em rooted (multi)graph} if $G$ is a locally finite, connected (multi)graph and $o$ is a distinguished vertex of $G$.
    \item We say that two multigraphs $G_1$ and $G_2$ are {\em isomorphic}, written as $G_1 \cong G_2$, if there exists a bijection $\phi \colon V(G_1) \to V(G_2)$ such that for any $v, w \in V(G_1)$, the number of edges between $v,w$ in $G_1$ equals the number of edges between $\phi(v), \phi(w)$ in $G_2$.
    \item We say that the rooted (multi)graphs $(G_1, o_1)$ and $(G_2, o_2)$ are {\em rooted isomorphic} if there exists a graph-isomorphism between $G_1$ and $G_2$ that maps $o_1$ to $o_2$.
    \item For $r \in \mathbb{N}$, we define $B_{r}^{\sss(G)}(o)$, the (closed) $r$-ball around $o$ in $G$ or $r$-neighborhood of $o$ in $G$, as the subgraph of $G$ spanned by all vertices of graph distance at most $r$ from $o$. We think of $B_{r}^{\sss(G)}(o)$ as a rooted (multi)graph with root $o$.
\end{enumerate}
\end{defn}
We next define a metric space over the set of rooted (multi)graphs:
\begin{defn}[Metric space on rooted (multi)graphs]\label{def: Metric Space}
  Let $\mathscr{G}_\star$ be the space of all rooted (multi)graphs. Let $(G_1, o_1)$ and $(G_2, o_2)$ be two rooted (multi)graphs. Let $R^\star = \sup\left\{r\geq 0| B_{r}^{\sss(G_1)}(o_1) \cong B_{r}^{\sss(G_2)}(o_2) \right\}.$ The metric $d_{\sss \mathscr{G}_\star}\colon \mathscr{G}_\star \times \mathscr{G}_\star \to \mathbb{R}_{\geq 0}$ is defined by
 \begin{equation}
     d_{\sss \mathscr{G}_\star}\left((G_1, o_1), (G_2, o_2)\right) := \frac{1}{R^\star+1}.
 \end{equation}
\end{defn}
We next define local weak convergence and local convergence in probability for any sequence from the space of rooted random graphs:

\begin{defn}[Local weak convergence of rooted (multi)graphs]
     The sequence of random (multi)graphs $(G_n)_{n\geq 1}$ is said to {\em converge locally weakly} to the random rooted graph $(G, o)$, a random variable taking values in $\mathscr{G}_\star$, having law $\mu$, if for every bounded continuous function $f\colon \mathscr{G}_\star \to \mathbb{R}$, as $n\to \infty$, 
     \begin{equation}
         \lim_{n\to \infty}\mathbb{E}[f(G_n, o_n)] =
         \mathbb{E}_{\mu}[f(G, o)],
     \end{equation}
     where $o_n$ is a uniformly chosen vertex from the vertex set of $G_n$. 
     \hfill\ensymboldefinition
\end{defn}

\begin{defn}[Local convergence in probability]
     The sequence of random (multi)graphs $(G_n)_{n\geq 1}$ is said to {\em converge locally in probability} to the random rooted graph $(G, o)$,  a random variable taking values in $\mathscr{G}_\star$, having law $\mu$, when, for every $r > 0$, and for every $(H, o') \in \mathscr{G}_\star$, as $n\to \infty$,
\begin{equation}
    \frac{1}{|V(G_n)|}\sum_{v\in V(G_n)}\mathbbm{1}_{\left\{B_r^{\sss(G_n)}(v)\cong (H, o')\right\}} \overset{\sss\prob}{\to} \mu(B_r^{\sss(G)}(o)\cong (H,o')).
\end{equation}    
\hfill\ensymboldefinition
\end{defn}

\subsection{Local convergence proofs}

In this section, we prove the results stated in Section \ref{sec-vertex-removal-local}. Let us take a strictly local centrality measure $R$. This means that it only depends on a finite fixed neighborhood, say $N\in \mathbb{N}$. The following lemma shows that the map $(G, o)\to (G(R, r), o)$ is continuous:

\begin{lem}[Continuity of centrality-based vertex removal]\label{Lemma}
Let $R$ be a strictly local centrality measure and $h\colon \mathscr{G}_\star \to \mathbb{R}$ be a continuous bounded map, where $\mathscr{G}_\star$ is the space of rooted graphs, with the standard metric, $d$, as in Definition \ref{def: Metric Space}. Suppose $f:\mathscr{G}_\star \to \mathbb{R}$ satisfies
\begin{equation}
    f(G, o) = h(G(R, r), o).
\end{equation}
Then $f(\cdot)$ is a bounded continuous function.
\end{lem}
\begin{proof} $f(\cdot)$ is bounded because $h(\cdot)$ is bounded. Thus, we only need to show that $f(\cdot)$ is continuous, i.e., for all $\varepsilon > 0$, there exists $\delta > 0$ such that 
$$d_{\sss\mathscr{G}_\star}\left((G_1, o_1), (G_2, o_2)\right) < \delta \implies |h(G_1(R,r), o_1) - h(G_2(R, r), o_2)| < \varepsilon.$$
Recall that $d_{\sss\mathscr{G}_\star}$ denotes the metric on $\mathscr{G}_\star$. 
By continuity of $h(\cdot)$, for all $\varepsilon > 0$, there exists $\delta' > 0$ such that 
    $$d_{\sss\mathscr{G}_\star}((G_1, o_1), (G_2, o_2)) < \delta' \implies |h(G_1, o_1) - h(G_2, o_2)| < \varepsilon.
    $$
Thus, it is enough to show that, for all $\delta'>0$, there exists $\delta>0$, such that
\begin{align} \label{eq: implication dstar}
    d_{\sss\mathscr{G}_\star}((G_1, o_1), (G_2, o_2)) < \delta \implies d_{\sss\mathscr{G}_\star}((G_1(R, r), o_1), (G_2(R, r), o_2)) < \delta'.
\end{align}
Let
\begin{equation}
    \delta = \left(\frac{\delta'}{1+(2N-1)\delta'}\right)^ \frac{1}{2N+3},
\end{equation}
then we claim that \eqref{eq: implication dstar} holds. In this we have used the definition of metric space over rooted graphs along with the fact that
\begin{equation}
    B_{2N+l}^{\sss(G_1)}(o_1) \overset{\sim}{=} B_{2N+l}^{\sss(G_2)}(o_2) \implies  B_{l}^{\sss(G_1(R, r))}(o_1) \overset{\sim}{=} B_{l}^{\sss(G_2(R,r))}(o_2). 
\end{equation} 
This is because the centrality measure $R(v)$ only depends on $B_N^{\sss(G)}(v)$. This proves the lemma.
\end{proof}

\begin{proof}[Proof of Theorem \ref{theorem-LC-vertex-removal}]
Suppose that $(G_n, o_n)$ converges locally weakly to $(G, o) \sim \mu$, which means that for all continuous and bounded functions $f : \mathscr{G}_\star \to \mathbb{R}$,
\begin{equation}\label{eq: lwc}
 \lim_{n\to \infty} \mathbb{E}[f(G_n, o_n)] = \mathbb{E}[f(G, o)].
\end{equation}
Using $f(\cdot)$ from Lemma \ref{Lemma} in \eqref{eq: lwc} we get
\begin{equation}
    \lim_{n\to \infty} \mathbb{E}[h(G_n(R, r), o_n)] = \lim_{n\to \infty} \mathbb{E}[f(G_n, o_n)] = \mathbb{E}[f(G, o)]= \mathbb{E}[h(G(R, r), o)].
\end{equation}
Thus, $(G_n(R, r), o_n)$ converges locally weakly to $(G(R, r), o)$. Similarly, suppose that $(G_n, o_n)$ converges locally in probability to $(G, o) \sim \mu$, which means for all continuous and bounded function $f \colon \mathscr{G}_\star \to \mathbb{R}$,
\begin{equation}\label{eq: lwc1}
  \mathbb{E}[f(G_n, o_n)| G_n] \overset{\sss\prob}{\to} \mathbb{E}[f(G, o)].
\end{equation}
Using $f(\cdot)$ from Lemma \ref{Lemma} in \eqref{eq: lwc1} we get 
\begin{equation}
     \mathbb{E}[h(G_n(R, r), o_n)|G_n] = \mathbb{E}[f(G_n, o_n)|G_n] \overset{\sss\prob}{\to} \mathbb{E}[f(G, o)]= \mathbb{E}[h(G(R, r), o)].
\end{equation} 
Thus, $(G_n(R, r), o_n)$ converges locally in probability to $(G(R, r), o)$.
\end{proof}

\begin{proof}[Proof of Corollaries \ref{upperbound for proportion} and \ref{connected component general}]
        These proofs follow immediately from the upper bound on the giant given in \cite[Corollary 2.28]{Hofst23}, and the convergence of the number of connected components in \cite[Corollary 2.22]{Hofst23}.
\end{proof}

\section{Proofs: Degree centrality and configuration model}

In this section, we restrict ourselves to the configuration model and investigate all possible $\alpha$-proportion vertex removals, which can be performed with respect to degree centrality, and then compare them on the basis of the size of giant.

\subsection{Giant in vertex removed configuration model: Proof of Theorem \ref{giant component condition and the proportion of giant}}\label{subsec- proportion of giant}
For the proof we use the vertex explosion construction as in Section \ref{Intermediate construction}. We apply the following steps on $G_n$:
\begin{enumerate}
 \item Choose $\lfloor n_k r_k^{\sss(n)}\rfloor$ vertices uniformly from the set of vertices of degree $k$ and explode them into $k$ vertices of degree $1$.
 \item Label these newly formed degree-$1$ vertices red.
 \item Repeat this for each $k\in\mathbb{N}$.
 \end{enumerate}
This newly constructed graph is $\tilde{G}(V')$ from Section \ref{Intermediate construction} for an appropriately chosen $V'$. We shall denote it by $\tilde{G}_n$. Let $\tilde{n}_k $, $\tilde{n}$ and $n_+$ denote the number of degree-$k$ vertices, total number of vertices, and number of vertices with red label in the vertex-exploded graph $\tilde{G}_n$, respectively. Then,
\begin{equation}\label{n_k tilde}
     \tilde{n}_i =  \left.
                        \begin{cases}
                         n_+ +n_1-\lfloor n_1 r^{\sss(n)}_1 \rfloor & \text{for }i = 1, \\
                         n_i - \lfloor n_i r^{\sss(n)}_i \rfloor & \text{for }i\neq 1.
                       \end{cases}
                       \right.
 \end{equation}

\begin{equation}\label{n +} 
    n_+ =  \sum_{i=1}^{\infty}i \lfloor n_i r^{\sss(n)}_i\rfloor = n\sum_{i=1}^{\infty}i\frac{ \lfloor n_i r^{\sss(n)}_i\rfloor}{n}.
\end{equation}

\begin{equation}\label{n tilde} 
    \tilde{n} =  \sum_{i=1}^{\infty} \tilde{n}_k = n + \sum_{i=1}^{\infty}(i-1) \lfloor n_i r^{\sss(n)}_i\rfloor = n(1+\sum_{i=1}^{\infty}(i-1) \frac{\lfloor n_i r^{\sss(n)}_i\rfloor}{n}) .
\end{equation}
Then, by applying the dominated convergence theorem on \eqref{n +} and \eqref{n tilde},
     \begin{equation}\label{n + probability approximation}
      n_+ =   \mathbb{E}[D r_D]n + o(n),\qquad
      \tilde{n} = \beta_{\boldsymbol{r}} n + o(n).
     \end{equation}
Recall that $\beta_{\boldsymbol{r}} = \mathbb{E}[D r_D] + 1 - \alpha$. Due to \eqref{n_k tilde} and \eqref{n + probability approximation}, 
\begin{equation}\label{abcd}
    \frac{\tilde{n}_i}{n} = \left.
                    \begin{cases}
                    \mathbb{E}[D r_D] +p_1(1-r_1)+ o(1) & \text{for }i = 1, \\
                    p_i(1-r_i)+ o(1) & \text{for }i\neq 1.  
                    \end{cases}
                      \right. 
\end{equation}
We define
\begin{equation}\label{p_j tilde}
     \tilde{p}_i := \left.
                    \begin{cases}
                    \beta_{\boldsymbol{r}}^{-1}(\mathbb{E}[D r_D] +p_1(1-r_1)) & \text{for }i = 1, \\
                    \beta_{\boldsymbol{r}}^{-1}p_i(1-r_i) & \text{for }i\neq 1.  
                    \end{cases}
                      \right.
 \end{equation} 
Then, by \eqref{abcd}, for every $i\geq 1$,
 \begin{equation}\label{n_k tilde probability approximation}
     \frac{\tilde{n}_i}{\tilde{n}} \to \tilde{p}_i.
 \end{equation}
Let $\tilde{D}$ be a random variable with probability mass function $(\tilde{p}_j)_{j\geq 1}$.
 
By Lemma \ref{still a configuration model after vertex explosion with respect to degree}, $\tilde{G}_{n}$ is also a configuration model. Equation \eqref{n_k tilde probability approximation} shows that the regularity Condition \ref{cond-degrees}(a) holds for $\boldsymbol{\tilde{d}}$. Similarly, Condition \ref{cond-degrees}(b) also holds. Using Theorem \ref{lwc}, the giant in $\tilde{G}_{n}$ exists if and only if $\mathbb{E}[\tilde{D}(\tilde{D}-2)] > 0$. We compute
 \begin{equation}\label{expectations imp}
     \begin{split}
         \mathbb{E}[\tilde{D}(\tilde{D}-2)]& = \frac{1}{\beta_{\boldsymbol{r}}}\sum_{j\geq 1}j(j-2)p_j(1-r_j) -\frac{1}{\beta_{\boldsymbol{r}}}\mathbb{E}[Dr_D]\\
         &= \frac{1}{\beta_{\boldsymbol{r}}}\mathbb{E}[D(D-2)(1-r_D)]- \frac{1}{\beta_{\boldsymbol{r}}}\mathbb{E}[Dr_D]\\
         &=\frac{1}{\beta_{\boldsymbol{r}}}\mathbb{E}[D(D-1)(1-r_D)]- \frac{1}{\beta_{\boldsymbol{r}}}\mathbb{E}[D].
     \end{split}
 \end{equation}
Recall $\nu_{\boldsymbol{r}}$ from \eqref{condition for giant}. By \eqref{expectations imp},
\begin{equation}\label{eq: equivalence of nur}
    \mathbb{E}[\tilde{D}(\tilde{D}-2)] > 0\iff \nu_{\boldsymbol{r}} > 1.
\end{equation}
Thus, $\Tilde{G}_{n}$ has a giant if and only if $\nu_{\boldsymbol{r}}>1$. Recall that the removal of red vertices (having degree 1) from $\Tilde{G}_n$ gives $G_{n,\boldsymbol{r}^{\sss(n)}}$. Since we remove a fraction $n_+/\tilde{n_1}$ of all degree-1 vertices, due to the law of large numbers for a hypergeometric distribution, we asymptotically remove the same fraction of degree-1 vertices  in the giant component. Thus, as $n\to \infty$,
\begin{equation}\label{equation comparison}
    v(C_1(\boldsymbol{r}^{\sss(n)})) = v(\tilde{C_1}) - v_1(\tilde{C_1}) \times \frac{n_+}{\tilde{n_1}}(1+o_{\sss\mathbb{P}}(n)). 
\end{equation}
See Janson \cite[(3.5)]{janson2009percolation}. Also, by Theorem \ref{lwc},
\begin{align}\label{important limits}
    \frac{v(\tilde{C_1})}{\tilde{n}} \overset{\sss\prob}{\to} 1-g_{\tilde{D}}(\eta_{\boldsymbol{r}})\qquad \text{ and }\qquad
    \frac{v_1(\tilde{C_1})}{\tilde{n}} \overset{\sss\prob}{\to} \tilde{p}_1(1-\eta_{\boldsymbol{r}}),
\end{align}
where $\eta_{\boldsymbol{r}} \in (0,1]$ satisfies $\beta_{\boldsymbol{r}} g_{\tilde{D}}'(\eta_{\boldsymbol{r}}) = \expec[D]\eta_{\boldsymbol{r}}$. Here, we recall that $v(C_1(\boldsymbol{r}))$ is the number of vertices in the largest component of $G_{n,\boldsymbol{r}}$. We define $v(\Tilde{C}_1)$ to be the number of vertices in the largest component of $\tilde{G}_n$ and $v_1(\tilde{C}_1)$ to be the number of degree-1 vertices in the largest component of $\Tilde{G}_n$.

Due to \eqref{equation comparison}, we conclude that there is a giant in $G_{n, \boldsymbol{r}^{\sss(n)}}$ if and only if there is a giant in $\Tilde{G}_n$. Thus, the giant exists if $\nu_{\boldsymbol{r}}>1$ and does not exist if $\nu_{\boldsymbol{r}} \leq 1$. By \eqref{n + probability approximation}, $\tilde{n}/n \to \beta_{\boldsymbol{r}}$. This means that \eqref{important limits} and \eqref{equation comparison} give us
\begin{equation}\label{another expression for rho}
    \frac{v(C_1(\boldsymbol{r}^{\sss(n)}))}{n} \overset{\sss\prob}{\to} \beta_{\boldsymbol{r}}(1-g_{\tilde{D}}(\eta_{\boldsymbol{r}})) -(1-\eta_{\boldsymbol{r}})\expec[Dr_D].
\end{equation}
Let
$$\rho(\boldsymbol{r}) := \beta_{\boldsymbol{r}} (1-g_{\tilde{D}}(\eta_{\boldsymbol{r}})) -(1-\eta_{\boldsymbol{r}})\mathbb{E}[Dr_D] = 1- \alpha-\mathbb{E}[Dr_D]\eta_{\boldsymbol{r}} -\beta_{\boldsymbol{r}} g_{\tilde{D}}(\eta_{\boldsymbol{r}}).$$ 
By definition of generating function,
$$\beta_{\boldsymbol{r}} g_{\tilde{D}}(\eta_{\boldsymbol{r}}) = \sum_{j\geq 1}\tilde{p}_j\eta_{\boldsymbol{r}}^j = \sum_{j\geq1}p_i(1-r_i)\eta_{\boldsymbol{r}}^i +\mathbb{E}[Dr_D]\eta_{\boldsymbol{r}}.$$
Thus, 
$$\rho(\boldsymbol{r})= 1-\alpha -2\mathbb{E}[Dr_D]\eta_{\boldsymbol{r}} -\sum_{i\geq 1}p_i(1-r_i)\eta_{\boldsymbol{r}}^i.$$
Removing a red vertex removes exactly one edge. Hence the number of edges in the giant component of the vertex-removed graph is the number of edges in the vertex-exploded graph minus the number of red vertices in the giant component of the vertex-exploded graph. Now using law of large numbers, as in \eqref{equation comparison}, we get
\begin{equation}\label{edges limit}
    e(C_1(\boldsymbol{r}^{\sss(n)}))  = e(\tilde{C}_1) -v_1(\tilde{C_1}) \times \frac{n_+}{\tilde{n_1}}(1+o_{\sss\mathbb{P}}(n)).
\end{equation}
Dividing by $n$ and taking limits in probability on both sides of \eqref{edges limit} completes the proof of Theorem \ref{giant component condition and the proportion of giant}.
\qed


\begin{proof}[Proof of Corollary \ref{alpha critical and size of giant k alpha}]
Let $k^{\sss(n)}$ be the top $\alpha$-quantile for the degree distribution $D_n$, which has probability mass function $(n_j/n)_{j\geq 1}$. Define $\Bar{{\boldsymbol{r}}}^{\sss(n)}(\alpha)$ to be
$$\Bar{{\boldsymbol{r}}}^{\sss(n)}(\alpha) =\left(0,0,0,\ldots0,\frac{n(\alpha-\mathbb{P}(D_n>k^{\sss(n)}))}{n_{k^{\sss(n)}}},1,1,\ldots\right).$$
Due to Theorem \ref{giant component condition and the proportion of giant}, it is enough to show that $(\bar{\boldsymbol{r}}^{\sss(n)}(\alpha))_{n\geq 1} \in S(\bar{\boldsymbol{r}}(\alpha))$. This follows by Conditions \ref{cond-degrees}.
The proof for the bottom $\alpha$-quantile is identical. This completes the proof.
\end{proof}

\subsection{Preliminary bounds}\label{sec-bounds and inequalities}

In this section, we prove bounds on the size of the giant in Theorem \ref{bound for rho}. We also discuss a bound on $\eta_{\boldsymbol{r}}$, interpreted as the half-edge extinction probability. Recall from \eqref{eta satisfies this}, $\eta_{\boldsymbol{r}}$ is the smallest solution in $(0,1]$ of \eqref{eta satisfies this}.
 \begin{lem}[Lower bound for $\eta_{\boldsymbol{r}}$] \label{theta lower bound lemma} For every $\alpha$-sequence $\boldsymbol{r}$,
 \begin{equation}\label{theta lower bound}
     \eta_{\boldsymbol{r}} \geq \frac{\mathbb{E}[Dr_D]}{\mathbb{E}[D]}.
 \end{equation}
 \end{lem}
 \begin{proof} 
 By \eqref{eta satisfies this} and \eqref{g_r},
 $$ \sum_{j\geq 1}j p_j (1-r_j)\eta_{\boldsymbol{r}}^{j-1}+ \mathbb{E}[D r_D] = \eta_{\boldsymbol{r}} \mathbb{E}[D],$$
 which implies that
 $$\mathbb{E}[D(1-r_D)\eta_{\boldsymbol{r}}^{D-1}] = \mathbb{E}[(\eta_{\boldsymbol{r}}-r_D)D].$$
 In the above equality, obviously LHS $\geq 0$. Thus, RHS $\geq 0$ as required.
\end{proof}
\begin{rem}[Lower bound for $\eta$ dependent only on $\alpha$]\label{trivial lower bound for eta}
    The lower bound given above is non-trivial, since using $D\geq 1$ a.s. gives
    \begin{equation}
       \eta_{\boldsymbol{r}} \geq \frac{\mathbb{E}[Dr_D]}{\mathbb{E}[D]} \geq \frac{\alpha}{\expec[D]} > 0.
    \end{equation}
This lower bound is independent of the $\alpha$-sequence used. \hfill\ensymboldefinition
\end{rem}
\begin{proof}[Proof of Theorem \ref{bound for rho}]
By \eqref{rho}, Lemma \ref{theta lower bound lemma} and Remark \ref{trivial lower bound for eta},
\begin{align}
    \rho({\boldsymbol{r}}) &= 1-\alpha-2\mathbb{E}[Dr_D]\eta_{\boldsymbol{r}} - \sum_{j\geq 1}p_i(1-r_i)\eta_{\boldsymbol{r}}^i \\
    & \leq 1 -\alpha - 2\mathbb{E}[Dr_D]\eta_{\boldsymbol{r}} \leq 1 -\alpha - \frac{2(\mathbb{E}[Dr_D])^2}{\expec[D]} . \label{eq: inequality for rho}
\end{align}
Since $D \geq 1$ almost surely, we have $\expec[Dr_D] \geq \expec[r_D] = \alpha$. Using this fact in \eqref{eq: inequality for rho} proves \eqref{upper bound for rho independent of eta}. If the $\alpha$-sequence, $\boldsymbol{r}$, is positively correlated with the degree sequence, we have $\expec[Dr_D]\geq \alpha \expec[D]$. Using this fact in \eqref{eq: inequality for rho} proves \eqref{eq: inequality improved for positively correlated alpha sequences}. The mean-value theorem implies that there exists $\zeta \in (\eta_{\boldsymbol{r}},1)$ such that $1-g_{\boldsymbol{r}}(\eta_{\boldsymbol{r}}) = g_{\boldsymbol{r}}'(\zeta)(1-\eta_{\boldsymbol{r}})$. Thus, also using \eqref{another expression for rho}, we obtain
\begin{equation}\label{rho is this}
    \rho({\boldsymbol{r}}) = \beta_{\boldsymbol{r}} (1-g_{\boldsymbol{r}}(\eta_{\boldsymbol{r}})) - \mathbb{E}[Dr_D](1-\eta_{\boldsymbol{r}}) =(\beta_{\boldsymbol{r}} g_{\boldsymbol{r}}'(\zeta) - \mathbb{E}[Dr_D])(1-\eta_{\boldsymbol{r}}).
\end{equation}
Since $g_{\boldsymbol{r}}(\cdot)$ is a generating function of a non-negative random variable, $g_{\boldsymbol{r}}'(\cdot)$ is an increasing function, so that
$$\frac{\expec[D]}{\beta_{\boldsymbol{r}}}\eta_{\boldsymbol{r}} =g_{\boldsymbol{r}}'(\eta_{\boldsymbol{r}})\leq g_{\boldsymbol{r}}'(\zeta)\leq g_{\boldsymbol{r}}'(1) = \frac{\expec[D]}{\beta_{\boldsymbol{r}}}.$$
Substituting this inequality in \eqref{rho is this}, we get
\begin{equation}\label{eq: bound on giant using mvt}
    \mathbb{E}[D(\eta_{\boldsymbol{r}}-r_D)](1-\eta_{\boldsymbol{r}})\leq\rho({\boldsymbol{r}}) \leq\mathbb{E}[D(1-r_D)](1-\eta_{\boldsymbol{r}}).
\end{equation}
Using the lower bound on $\eta_{\boldsymbol{r}}$ (Lemma \ref{theta lower bound lemma}) in \eqref{eq: bound on giant using mvt}, we obtain bound for $\rho(\boldsymbol{r})$, in \eqref{upper bound for e independent of eta}. Next, we obtain bound for $e(\boldsymbol{r})$. By Theorem \ref{giant component condition and the proportion of giant},
$$e({\boldsymbol{r}})=\left( \frac{\mathbb{E}[D]}{2}(1+\eta_{\boldsymbol{r}})-\mathbb{E}[Dr_D] \right)(1-\eta_{\boldsymbol{r}}) = -\frac{\mathbb{E}[D]}{2}\eta_{\boldsymbol{r}}^2+\mathbb{E}[Dr_D]\eta_{\boldsymbol{r}} + \frac{\mathbb{E}[D]}{2}-\mathbb{E}[Dr_D].$$
Define the polynomial $P(\cdot)$ by $$P(x):=-\frac{\mathbb{E}[D]}{2}x^2+\mathbb{E}[Dr_D]x + \frac{\mathbb{E}[D]}{2}-\mathbb{E}[Dr_D].$$
Note that $P(\cdot)$ is a quadratic polynomial with a negative leading coefficient. Thus, it is maximal at $x^\star = \mathbb{E}[Dr_D]/\mathbb{E}[D] \in(0,1)$. By Lemma \ref{theta lower bound lemma}, 
$$\eta_{\boldsymbol{r}} \in \left[\frac{\mathbb{E}[Dr_D]}{\mathbb{E}[D]}, 1\right].$$ 
Thus,
$$0=P(1) < e({\boldsymbol{r}}) = P(\eta_{\boldsymbol{r}})\leq P(x^\star).$$ 
This gives \eqref{upper bound for e independent of eta}, which completes the proof.
\end{proof}
\subsection{$\varepsilon$-transformations and half-edge extinction probability} 

In this section, we study the effect of a small perturbation in the $\alpha$-sequence on the size of the giant after vertex removal. For this we define the notion of an `$\varepsilon$-transformation' that moves mass to the left:
 
\begin{defn}[$\varepsilon$-transformation]\label{def-varepsilon-transformation}
Suppose ${\boldsymbol{r}} = (r_j)_{j\geq 1}$ is an $\alpha$-sequence. Fix $k$,$l$ $\in \mathbb{N}$. We call ${\boldsymbol{r}}^{k,l}(\varepsilon)= (r^{k,l}_j(\varepsilon))_{j \geq 1}$ an $\varepsilon$-transformation on the co-ordinates $(k, k+l)$ if it satisfies that $r^{k,l}_i(\varepsilon) =  r_i$ for $i \notin \{k, k+l\}$, while
    \begin{equation}
       r^{k,l}_k(\varepsilon) =   r_k + \frac{\varepsilon}{p_k}
       \qquad \text{and} \qquad
      r^{k,l}_{k+l}(\varepsilon) =  r_{k+l} - \frac{\varepsilon}{p_{k+l}}.
    \end{equation}
Note that ${\boldsymbol{r}}^{k,l}(\varepsilon)$ is also an $\alpha$-sequence for $\varepsilon \leq p_{k+l}r_{k+l}.$ Thus, the domain of $\varepsilon$ is   $[0, p_{k+l}r_{k+l}]$.
\hfill\ensymboldefinition
\end{defn} 
Let $\rho^{k,l}(\varepsilon)$ denote the limiting proportion of vertices in the giant after the vertex removal procedure according to ${\boldsymbol{r}}^{k,l}(\varepsilon)$.
\begin{rem}[Supercritical regime]
    Throughout this section we assume that we start with the supercritical regime, i.e., $\nu_{\boldsymbol{r}}>1$, recall $\nu_{\boldsymbol{r}}$ from \eqref{condition for giant}.
\hfill\ensymboldefinition
\end{rem}
\begin{rem}[Graph remains supercritical after $\varepsilon$-transformation]\label{lem-epsilonsupercritical} Starting in the supercritical regime, the sequence of random graphs remain in the supercritical regime even after the $\varepsilon$-transformation. In other words, 
$\nu_{\boldsymbol{r}} >1 \implies \nu_{\boldsymbol{r}^{k,l}(\varepsilon)} > 1.$ Indeed, we can compute that
    \begin{align*}
        \nu_{\boldsymbol{r}^{k,l}(\varepsilon)} &= \frac{\expec[D(D-1)(1-r_D^{k,l}(\varepsilon))]}{\expec[D]}\\
        &=  \frac{\expec[D(D-1)(1-r_D)]}{\expec[D]} + \frac{k(k-1)(-\varepsilon) + (k+l)(k+l-1)\varepsilon}{\expec[D]}>\nu_{\boldsymbol{r}}.
    \end{align*}
\hfill\ensymboldefinition
\end{rem}
\subsubsection{Monotonicity of the half-edge extinction probability}

In this section, we investigate the effect of the $\varepsilon$-transformation on the half-edge extinction probability of the configuration model after vertex explosion. Let $\tilde{D}$ with probability mass function $(\tilde{p}_{j})_{j\geq 1}$ be as described in Section \ref{subsec- proportion of giant}. $\tilde{D}$ was computed for the $\alpha$-sequence $\boldsymbol{r}$. Let $\tilde{D}_\varepsilon$ be the corresponding $\tilde{D}$ for the $\alpha$-sequence $\boldsymbol{r}^{k,l}(\varepsilon)$. Let $g_\varepsilon(\cdot)$ and $g(\cdot)$ be the generating functions of $\tilde{D}_\varepsilon$ and $\tilde{D}$, respectively. Let $\eta_\varepsilon,\eta \in (0,1)$ satisfy  
\begin{align}
    g_\varepsilon'(\eta_\varepsilon) = \frac{\mathbb{E}[D]}{\beta_\varepsilon} \eta_\varepsilon
    \qquad\text{ and }
    \qquad g'(\eta) = \frac{\mathbb{E}[D]}{\beta} \eta,
\end{align} 
where
\begin{align}\label{value of beta}
    \beta_\varepsilon = \mathbb{E}[Dr^{k,l}_D(\varepsilon)] + 1 -\alpha\qquad  \text{ and }
    \qquad \beta = \mathbb{E}[Dr_D]+1-\alpha. 
\end{align} 
Then $\eta$ and $\eta_\varepsilon$ are the half-edge extinction probabilities in the configuration model with limiting degree distributions $\tilde{D}$ and $\tilde{D}_\varepsilon$, respectively. By the definition of $r_D^{k,l}(\varepsilon)$, 
\begin{equation}\label{beta and beta epsilon relation}
    \mathbb{E}[Dr^{k,l}_D(\varepsilon)] =\mathbb{E}[Dr_D] -l\varepsilon.
\end{equation}
This means that $\beta_\varepsilon= \beta - l\varepsilon$. We can also relate $g_\varepsilon$ and $g$ by noting that
\begin{align*}
    g_\varepsilon(s) = \mathbb{E}[s^{\tilde{D}_\varepsilon}] = \beta_\varepsilon ^{-1}(\sum_{j\geq 1}&p_j(1-r^{k,l}_j(\varepsilon))s^j + \mathbb{E}[Dr^{k,l}_D(\varepsilon)]s),
\end{align*}
so that
\begin{equation*}
     \beta_\varepsilon g_\varepsilon(s)= \sum_{j\geq 1}p_j(1-r_j)s^j + \mathbb{E}[Dr_D]s - l\varepsilon s +\sum_{j\in \{k, k+l\}}p_j(r_j - r^{k,l}_j(\varepsilon))s^j.
\end{equation*}
In turn, this means that
\begin{equation}\label{g varepsilon and g relation}
     \beta_\varepsilon g_\varepsilon(s) = \beta g(s) -\varepsilon (s^k - s^{k+l} + ls).
\end{equation}
Recall that the proportion of vertices in the giant component of the graph obtained after vertex removal with respect to $\boldsymbol{r}^{k,l}(\varepsilon)$ is $\rho^{k,l}(\varepsilon)$. By \eqref{another expression for rho},
\begin{equation}\label{proportion --}
    \rho^{k,l}(\varepsilon) = 1-\alpha  + (\mathbb{E}[Dr_D]- l\varepsilon)\eta_\varepsilon -\beta_\varepsilon g_\varepsilon (\eta_\varepsilon).
\end{equation}
Substituting \eqref{g varepsilon and g relation} into \eqref{proportion --} gives
\begin{align}\label{proportion of giant as a function of epsilon}
    \rho^{k,l}(\varepsilon) = 1-\alpha - \beta g(\eta_\varepsilon) + \varepsilon(\eta_\varepsilon ^k - \eta_\varepsilon ^{k+l}) +\mathbb{E}[Dr_D]\eta_\varepsilon.
\end{align}
The following lemma shows that $\eta_\varepsilon$ is well-defined for $\varepsilon$ such that $\nu_{\boldsymbol{r}^{k,l}(\varepsilon)}>1$:
\begin{lem}[Uniqueness]\label{uniqueness of eta}
    Let $g(\cdot)$ be the generating function for $D$, let $p_1>0$, and let $\eta \in (0,1]$ satisfy
    \begin{equation}\label{eq: half edge extinction}
        g'(\eta) = \expec[D] \eta,
    \end{equation}
Then $\eta$ exists and is unique if $\nu =\mathbb{E}[D(D-1)]/{\mathbb{E}[D]} > 1.$ If $\nu\leq 1$, then $\eta =1$ is the only solution to \eqref{eq: half edge extinction}.
\end{lem}
The above lemma follows from Janson and Luczak \cite[Theorem 2.3]{JanLuc07}. For $p_1=0$, we define $\eta=0$ instead. Then, $\eta$ can be interpreted as the extinction probability of a branching process with offspring distribution $p_k^\star=(k+1)p_k/\expec[D]$. The next lemma is needed in the rest of this section:

 \begin{lem}[Negative polynomial]\label{positive polynomial}For $s \in (0,1)$,
 $$(k+l)s^{k+l-1}-ks^{k-1}-l< 0.$$
 \end{lem}
 \begin{proof}
Let $P(x) := (k+l)x^{k+l-1} - k x^{k-1} - l$, then
$$P'(x) = (k+l)(k+l-1)x^{k+l-2}-k(k-1)x^{k-2}.$$
This means that
    $$
    P'(x) > 0 \iff x^l >\frac{k(k-1)}{(k+l)(k+l-1)}. 
    $$ 
The only minimizer of $P(\cdot)$ in $(0,1)$, denoted by $\gamma$, equals
    $$
    \gamma = \left(\frac{k(k-1)}{(k+l)(k+l-1)}\right)^{1/l}. 
    $$
Since there is no maximum in the interval $(0, 1)$, the maximum occurs either in $0$ or $1$. We have $P(0)= -l$ and $P(1) = 0$, so that $P(x) < 0$ for all $x \in (0,1)$. 
 \end{proof}

In the next lemma, we compute the derivative of $\eta_\varepsilon$ with respect to $\varepsilon$, and show that it is negative: 
 
\begin{lem}[Half-edge extinction probability is decreasing in $\varepsilon$] \label{partial derivative of eta sign}Suppose $\nu_{\boldsymbol{r}}>1$. For all $\varepsilon \in [0, p_{k+l}r_{k+l}]$,
    $$
    \frac{\partial \eta_\varepsilon}{\partial \varepsilon} = \frac{(k+l)\eta_\varepsilon^{k+l-1} - k \eta_\varepsilon^{k-1} - l}{\expec[D] - \beta_\varepsilon g_\varepsilon''(\eta_\varepsilon)} < 0. 
    $$
\end{lem}

\begin{proof}
Differentiating \eqref{g varepsilon and g relation} gives
    \begin{equation}\label{eq: differentiation of g once}
    \beta_\varepsilon g_\varepsilon'(s) = \beta g'(s) - \varepsilon(k s^{k-1} - (k+l)s^{k+l-1} +l).
    \end{equation}
By definition of $\eta_\varepsilon$, we have $\beta_\varepsilon g_\varepsilon'(\eta_\varepsilon)  = \expec[D] \eta_\varepsilon$. Thus,
    \begin{equation}\label{eq: eta varepsilon relation with function}
    \beta g'(\eta_\varepsilon) - \varepsilon(k\eta_\varepsilon^{k-1} - (k+l)\eta_\varepsilon^{k+l-1} +l) = \expec[D] \eta_\varepsilon.
    \end{equation}
Define $h\colon \mathbb{R}^2 \to \mathbb{R}$, as
    \begin{equation}\label{eq: apply implicit}
    h(x, y) = \beta g'(x) - y(kx^{k-1}-(k+l)x^{k+l-1}+l)-\expec[D]x.
    \end{equation}
Due to Lemma \ref{positive polynomial}, for $x \in [0, 1)$, we have
\begin{equation}
    \frac{\partial h(x, y)}{\partial y} = (k+l)x^{k+l-1}-kx^{k-1}-l < 0.
\end{equation}
Also, $h(\eta_\varepsilon, \varepsilon)=0$ for every $\varepsilon>0$. The implicit function theorem \cite[Theorem 9.28]{MR0166310} on $h:\mathbb{R}^2\to \mathbb{R}$ implies that $\varepsilon\mapsto \eta_\varepsilon$ is differentiable. Thus, we differentiate \eqref{eq: eta varepsilon relation with function} with respect to $\varepsilon$, to obtain
\begin{align*}
    \expec[D] \frac{\partial\eta_\varepsilon}{\partial\varepsilon} = \beta g''(\eta_\varepsilon)\frac{\partial \eta_\varepsilon}{\partial \varepsilon} - (k&\eta_\varepsilon^{k-1} - (k+l)\eta_\varepsilon^{k+l-1} +l)  \\&
    -\varepsilon\big(k(k-1)\eta_\varepsilon^{k-2}-(k+l)(k+l-1)\eta_\varepsilon^{k+l-2}\big)\frac{\partial\eta_\varepsilon}{\partial\varepsilon},
\end{align*}
which gives
\begin{equation}\label{3}
    \frac{\partial \eta_\varepsilon}{\partial \varepsilon} = \frac{(k+l)\eta_\varepsilon^{k+l-1} - k \eta_\varepsilon^{k-1} - l}{\expec[D] - \beta g''(\eta_\varepsilon)+\varepsilon(k(k-1)\eta_\varepsilon^{k-2}-(k+l)(k+l-1)\eta_\varepsilon^{k+l-2})}.
\end{equation}
Differentiating \eqref{g varepsilon and g relation} twice w.r.t.\ $s$, and putting $s=\eta_\varepsilon$, we obtain
\begin{equation*}
\beta_\varepsilon g_\varepsilon''(\eta_\varepsilon) = \beta g''(\eta_\varepsilon) - \varepsilon(k(k-1)\eta_\varepsilon^{k-2} - (k+l)(k+l-1)\eta_\varepsilon^{k+l-2}).
\end{equation*}
This means that
\begin{equation}\label{2}
    \beta g''(\eta_\varepsilon) = \beta_\varepsilon g_\varepsilon''(\eta_\varepsilon) + \varepsilon(k(k-1)\eta_\varepsilon^{k-2} - (k+l)(k+l-1)\eta_\varepsilon^{k+l-2}).
\end{equation}
Substituting \eqref{2} in \eqref{3}, we get
\begin{equation}\label{derivative of eta}
   \frac{\partial \eta_\varepsilon}{\partial \varepsilon} = \frac{(k+l)\eta_\varepsilon^{k+l-1} - k \eta_\varepsilon^{k-1} - l}{\expec[D] - \beta_\varepsilon g_\varepsilon''(\eta_\varepsilon)},
\end{equation}
as required. We next show that $\frac{\partial\eta_\varepsilon}{\partial \varepsilon}<0$. Let $h(x) = \beta_\varepsilon g_\varepsilon'(x) - \expec[D] x.$ Notice that $h(\eta_\varepsilon) = h(1) = 0$. By Rolle's theorem, there exists $\phi \in (\eta_\varepsilon, 1)$ such that $h'(\phi) = 0.$ This implies that
 $$g_\varepsilon''(\phi) = \frac{\expec[D]}{\beta_\varepsilon}.$$
 Since $g_\varepsilon(\cdot)$ is a generating function, $g_\varepsilon''(\cdot)$ is increasing, so that
 $$\beta_\varepsilon g_\varepsilon''(\eta_\varepsilon) < \beta_\varepsilon g_\varepsilon''(\phi) = \expec[D]. $$
Thus,
\begin{equation}\label{denominator positive}
    \expec[D] - \beta_\varepsilon g_\varepsilon''(\eta) > 0,
\end{equation}  
which means that \eqref{derivative of eta} is a well-defined expression, which means that $\eta_\varepsilon$ is differentiable and the derivative is given by \eqref{derivative of eta}. By Lemma \ref{positive polynomial}, the numerator in \eqref{derivative of eta} is negative, which completes the proof.
\end{proof}
Lemma \ref{partial derivative of eta sign} immediately implies that $\varepsilon\mapsto \frac{\partial \eta_\varepsilon}{\partial \varepsilon}$ is continuous:
\begin{cor}[Continuity of the derivative]\label{derivative of eta is continuous}
 Suppose $\nu_{\boldsymbol{r}}>1$. $\frac{\partial \eta_\varepsilon}{\partial \varepsilon}$ is a continuous function.
\end{cor}

We next show that the derivative of the half-edge extinction probability is bounded, which is also a consequence of Lemma \ref{partial derivative of eta sign}:
\begin{cor}[Bound on the derivative] \label{derivative of eta is bounded}
         Suppose $\nu_{\boldsymbol{r}}>1$. There exists $\psi_{k,l} <0$ such that $\psi_{k,l}<\frac{\partial \eta_\varepsilon}{\partial \varepsilon}<0$ for all $\varepsilon \in [0, p_{k+l}r_{k+l}]$. 
\end{cor}
\begin{proof}
Let $h(\varepsilon)=\frac{\partial\eta_\varepsilon}{\partial\varepsilon}$. Notice that $\varepsilon$ can only take values in $[0, p_{k+l}r_{k+l} ]$, by the Definition \ref{def-varepsilon-transformation}. This means that the domain of function $h(\cdot)$ is compact. Further by Corollary \ref{derivative of eta is continuous}, $h(\cdot)$ is continuous. This means that $h(\cdot)$ is bounded below. By Lemma \ref{partial derivative of eta sign} we already know that $h(\varepsilon) < 0$ for all $\varepsilon$. Thus, there exists $\psi_{k,l} < 0$ such that $\psi_{k,l}<h(\varepsilon)<0$ for all $\varepsilon \in [0, p_{k+l}r_{k+l}]$.
\end{proof}

\subsection{Derivative of vertex/edge limiting proportion in the giant of $G_{n,\boldsymbol{r}}(\varepsilon)$ }
In this section, we compute the derivative of the limiting proportion of vertices $\varepsilon\mapsto \rho^{k,l}(\varepsilon)$ and edges $\varepsilon\mapsto e^{k,l}(\varepsilon)$ in the giant component in $G_{n,\boldsymbol{r}}(\varepsilon)$, and show that these derivatives 
are positive. In the next proposition, we compute $\frac{\partial e^{k,l}(\varepsilon)}{\partial\varepsilon}$ and its sign:
\begin{prop}[Monotonicity of edge proportion]\label{derivative of edge proportion is positive}Suppose $\nu_{\boldsymbol{r}}>1$. For all $\varepsilon \in [0, p_{k+l}r_{k+l}]$, and $\alpha$-sequence $\boldsymbol{r}$,
    \begin{equation}\label{eq: edge derivative}
    \frac{\partial e^{k,l}(\varepsilon)}{\partial\varepsilon} = -\frac{\partial \eta_\varepsilon}{\partial\varepsilon}\mathbb{E}[D(\eta_\varepsilon-r^{k,l}_D(\varepsilon))] +l(1-\eta_\varepsilon)>0.
\end{equation}
\end{prop}

\begin{proof}
By Theorem \ref{giant component condition and the proportion of giant},
\begin{equation}\label{edges proportion}
  e^{k,l}(\varepsilon)=\frac{\expec[D]}{2}(1-\eta_\varepsilon^2)-(\mathbb{E}[Dr_D]-l\varepsilon)(1-\eta_\varepsilon). 
\end{equation}
Differentiating \eqref{edges proportion} with respect to $\varepsilon$, we obtain 
\begin{align*}
    \frac{\partial e^{k,l}(\varepsilon)}{\partial\varepsilon} &= -\frac{\partial \eta_\varepsilon
    }{\partial\varepsilon}(\mathbb{E}[D(\eta_\varepsilon-r_D)]+l\varepsilon) +l(1-\eta_\varepsilon) \\
    &= -\frac{\partial \eta_\varepsilon
    }{\partial\varepsilon}\mathbb{E}[D(\eta_\varepsilon-r^{k,l}_D(\varepsilon))] +l(1-\eta_\varepsilon),
\end{align*} 
as required. By Lemma \ref{partial derivative of eta sign}, $\frac{\partial\eta_\varepsilon}{\partial\varepsilon} <0$. By Lemma \ref{theta lower bound lemma}, for the $\alpha$-sequence $\boldsymbol{r}^{k,l}(\varepsilon)$, 
$$\eta_\varepsilon \geq \frac{\mathbb{E}[Dr^{k,l}_D(\varepsilon)]}{\mathbb{E}[D]} \implies \mathbb{E}[D(\eta_\varepsilon-r^{k,l}_D(\varepsilon))] >0.$$
This completes the proof.
\end{proof}
In the next proposition, we compute $\frac{\partial \rho^{k,l}(\varepsilon)}{\partial\varepsilon}$:
 \begin{prop}[Derivative of vertex proportion]\label{derivation} 
Suppose $\nu_{\boldsymbol{r}}>1$. For all $\varepsilon \in [0, p_{k+l}r_{k+l}]$, and $\alpha$-sequence $\boldsymbol{r}$,
 \begin{equation}\label{eq: derivative of rho function}
     \frac{\partial \rho^{k,l}(\varepsilon)}{\partial \varepsilon} = -A_\varepsilon\frac{\partial \eta_\varepsilon}{\partial \varepsilon} + B_\varepsilon,
 \end{equation}
 where
 \begin{align}
     &A_\varepsilon = \mathbb{E}[D(\eta_\varepsilon-r^{k,l}_D(\varepsilon))] + (k\eta_\varepsilon^{k-1}-(k+l)\eta_{\varepsilon}^{k+l-1})\varepsilon,\label{eq: A epsilon}\\
     &B_\varepsilon = \eta_\varepsilon^k(1 - \eta_\varepsilon ^{l})+ (k\eta_\varepsilon^{k-1}-(k+l)\eta_{\varepsilon}^{k+l-1})\varepsilon.\label{eq: B epsilon} 
\end{align}
Consequently, 
$$\frac{\partial \rho^{k,l}(\varepsilon)}{\partial \varepsilon}\Big|_{\varepsilon=0} = - \mathbb{E}[D(\eta-r_D)] \frac{\partial \eta_\varepsilon}{\partial \varepsilon}\Big|_{\varepsilon=0} +\eta^k(1-\eta^l) > 0.$$
 
\end{prop}

\begin{proof}
 Taking a derivative of \eqref{proportion of giant as a function of epsilon} with respect to $\varepsilon$ gives
 \begin{align*}
      \frac{\partial \rho^{k,l}(\varepsilon)}{\partial \varepsilon} &= -\beta g'(\eta_\varepsilon) \frac{\partial \eta_\varepsilon}{\partial \varepsilon} + \varepsilon(k\eta_\varepsilon ^{k-1} - (k+l)\eta_\varepsilon^{k+l-1}) + (\eta_\varepsilon^k - \eta_\varepsilon ^{k+l}) + \mathbb{E}[Dr_D]\frac{\partial \eta_\varepsilon}{\partial \varepsilon}\\
       &= (\mathbb{E}[Dr_D]-\beta g'(\eta_\varepsilon)) \frac{\partial \eta_\varepsilon}{\partial \varepsilon} + \eta_\varepsilon^k(1 - \eta_\varepsilon ^{l}) + \varepsilon(k\eta_\varepsilon ^{k-1} - (k+l)\eta_\varepsilon^{k+l-1}).
 \end{align*}
  Thus,
  \begin{equation}\label{equation important unnamed}
      \frac{\partial \rho^{k,l}(\varepsilon)}{\partial \varepsilon}= (\mathbb{E}[Dr_D]-\beta g'(\eta_\varepsilon)) \frac{\partial \eta_\varepsilon}{\partial \varepsilon} + B_\varepsilon.
  \end{equation}
 By \eqref{eq: eta varepsilon relation with function}, we have
   \begin{align*}
       \beta g'(\eta_\varepsilon) &= \mathbb{E}[D] \eta_\varepsilon + \varepsilon (k\eta_\varepsilon^{k-1} - (k+l)\eta_\varepsilon^{k+l-1} + l)\\
       &= A_\varepsilon - \mathbb{E}[Dr^{k,l}_D(\varepsilon)]+l\varepsilon = A_\varepsilon - \mathbb{E}[Dr_D].
   \end{align*}
Substituting this value in \eqref{equation important unnamed} proves the claim. 
\end{proof} 

 
   
 

 
 



Proposition \ref{derivation} implies that $\frac{\partial \rho^{k,l}(\varepsilon)}{\partial \varepsilon}\Big|_{\varepsilon=0}>0$. We next extend this to all $\varepsilon$:

\begin{prop}[Monotonicity of vertex proportion]\label{derivative positive conditions}Suppose $\nu_{\boldsymbol{r}}>1$. For all $\varepsilon \in [0, p_{k+l}r_{k+l}]$,
 $$ \frac{\partial \rho^{k,l}(\varepsilon)}{\partial \varepsilon} > 0.$$
\end{prop}
\begin{proof}
By Lemma \ref{theta lower bound lemma}, we have
\begin{equation}\label{eq: inequality concerning A_epsolon}
    \mathbb{E}[D(\eta_\varepsilon-r^{k,l}_D(\varepsilon))] >0 \implies A_\varepsilon > (k\eta_\varepsilon^{k-1}-(k+l)\eta_{\varepsilon}^{k+l-1})\varepsilon,
\end{equation}
where $A_\varepsilon$ is as in \eqref{eq: A epsilon}. By Lemma \ref{partial derivative of eta sign}, $\frac{\partial \eta_\varepsilon}{\partial \varepsilon} < 0$. Thus, by Proposition \ref{derivation} and \eqref{eq: inequality concerning A_epsolon}, 
$$ \frac{\partial \rho^{k,l}(\varepsilon)}{\partial \varepsilon} > (1- \frac{\partial \eta_\varepsilon}{\partial \varepsilon})(k\eta_\varepsilon^{k-1}-(k+l)\eta_{\varepsilon}^{k+l-1})\varepsilon  + \eta_\varepsilon^k(1 - \eta_\varepsilon ^{l}).$$
We have $\eta_\varepsilon \in (0,1)$. By Lemma \ref{positive polynomial}, $k\eta_\varepsilon^{k-1}-(k+l)\eta_\varepsilon^{k+l-1}> -l.$ This means
$$\frac{\partial \rho^{k,l}(\varepsilon)}{\partial \varepsilon} > \eta_\varepsilon^k(1 - \eta_\varepsilon ^{l}) -l \varepsilon(1- \frac{\partial \eta_\varepsilon}{\partial \varepsilon}).$$
Thus,
\begin{equation}\label{equation of small nbd}
   \varepsilon< \frac{\eta_\varepsilon^k(1 - \eta_\varepsilon ^{l})}{l (1- \frac{\partial \eta_\varepsilon}{\partial \varepsilon})} \implies \frac{\partial \rho^{k,l}(\varepsilon)}{\partial \varepsilon} > 0.
\end{equation}
By Lemma \ref{partial derivative of eta sign} and Remark \ref{trivial lower bound for eta},
$$\frac{\alpha}{\expec[D]} \leq \eta_\varepsilon \leq \eta < 1.$$ 
By Corollary \ref{derivative of eta is bounded}, we know that there exists $\psi_{k,l} < 0$, such that 
$\psi_{k,l}<\frac{\partial\eta_\varepsilon}{\partial \varepsilon}<0.$
Define $\tau(\boldsymbol{r})$ as
    $$\tau(\boldsymbol{r}) := \frac{\alpha^k(1-\eta)}{\expec[D]^{k} l (1- \psi_{k,l})} < \frac{\eta_\varepsilon^k(1 - \eta_\varepsilon ^{l})}{l (1- \frac{\partial \eta_\varepsilon}{\partial \varepsilon})}.
    $$
Thus, $\rho^{k,l}(\varepsilon)$ has a positive derivative in $(0, \tau(\boldsymbol{r}))$, because of \eqref{equation of small nbd}. If one iterates this on the $\alpha$-sequence $\boldsymbol{r}'= \boldsymbol{r}^{k,l}(\tau(\boldsymbol{r}))$, we get that $\rho^{k,l}(\varepsilon)$ has a positive derivative in $(0, \tau(\boldsymbol{r})+\tau(\boldsymbol{r}'))$. Next, we show $\tau(\boldsymbol{r})\leq \tau(\boldsymbol{r}')$. For that we note the following facts:
\begin{enumerate}
    \item Both $\boldsymbol{r}$ and $\boldsymbol{r}'$ have the same $\alpha$, while $\expec[D]$ is a constant.
    \item We are using the same $k, l$ for both $\varepsilon$-transformations $\boldsymbol{r}$ and $\boldsymbol{r}'$, which means that the exponents in the expression of $\tau(\cdot)$ are the same.
    \item $\psi_{k,l}$ depends only on $k$ and $l$, and is thus the same for both cases.
\end{enumerate}
 This means that $\tau(\boldsymbol{r})$ and $\tau(\boldsymbol{r}')$ differ only because of their dependence on $\eta=\eta(\boldsymbol{r})$ and $\eta=\eta(\boldsymbol{r}')$. Further, $\varepsilon\mapsto \eta_\varepsilon$ is a decreasing function. Thus, $\tau(\boldsymbol{r}')\geq \tau(\boldsymbol{r})$. As a result, $\rho^{k,l}(\varepsilon)$ has a positive derivative in the neighborhood $(0, 2\tau(\boldsymbol{r}))$. Iterate this to cover the full domain of the function. 
\end{proof}

\subsubsection{Proof of main result}
In this section, we prove Theorem \ref{theorem-comparison-rs}. We start by showing that we can go from $\boldsymbol{r}$ to $\boldsymbol{r}'$ by using $\varepsilon$-transformations when $\boldsymbol{r}'\preccurlyeq_{\boldsymbol{p}} \boldsymbol{r}$:
\begin{lem}[Stochastic ordering and $\varepsilon$-transformations]\label{characterization of alpha sequences order}
    Let ${\boldsymbol{r}}=(r_j)_{j\geq 1}$ and ${\boldsymbol{r}}'=(r'_j)_{j\geq 1}$ be two $\alpha$-sequences such that ${\boldsymbol{r}}' \preccurlyeq_{\boldsymbol{p}} {\boldsymbol{r}}$. Then ${\boldsymbol{r}}'$ can be obtained from ${\boldsymbol{r}}$ through a series of $\varepsilon$-transformations. 
\end{lem}

\longversion{The proof of Lemma \ref{characterization of alpha sequences order}, which is straightforward yet technically involved, is deferred to Appendix \ref{app-B}.}
\shortversion{The proof of Lemma \ref{characterization of alpha sequences order}, which is straightforward yet technically involved, is deferred to \cite[Appendix B]{RvdHMP}.}
With these tools, we can now prove our main result in Theorem \ref{theorem-comparison-rs}:
\begin{proof}[Proof of Theorem \ref{theorem-comparison-rs}] 
Recall $\bar{\boldsymbol{r}}(\alpha)$, and $\underline{\boldsymbol{r}}(\alpha)$ from Definition \ref{definition: upper and lower removal alpha sequence}. Let $\boldsymbol{r}=(r_i)_{i\geq1}$ be an arbitrary $\alpha$-sequence. Then, by definition of stochastic dominance,
$\underline{\boldsymbol{r}}(\alpha)\preccurlyeq_{\boldsymbol{p}} \boldsymbol{r} \preccurlyeq_{\boldsymbol{p}} \bar{\boldsymbol{r}}(\alpha).$
By Propositions \ref{derivative of edge proportion is positive} and \ref{derivative positive conditions}, the proportion of edges and vertices in the giant increases after each $\varepsilon$-transformation. Thus, we get the desired result from Lemma \ref{characterization of alpha sequences order}.
\end{proof}

We next prove Corollary \ref{general comparison}, which relies on the following lemma that allows us to compare $\alpha$-sequences with different $\alpha$'s:

\begin{lem}[Stochastic ordering]\label{Comparison in two arbitrary sequences of [0,1]}
    Fix $\varepsilon>0$. Suppose $\boldsymbol{r}$ is an $\alpha$-sequence and $\boldsymbol{r}'$ is an $(\alpha+ \varepsilon)$-sequence such that $\boldsymbol{r}\preccurlyeq_{\boldsymbol{p}} \boldsymbol{r}'$. Then there exists an $\varepsilon$-sequence, say $\boldsymbol{\delta} = (\delta_i)_{i\geq 1}$ satisfying $\boldsymbol{r}+\boldsymbol{\delta} \preccurlyeq_{\boldsymbol{p}} \boldsymbol{r}'$.
\end{lem}
\longversion{The proof of Lemma \ref{Comparison in two arbitrary sequences of [0,1]} is again deferred to Appendix \ref{app-B}.}
\shortversion{The proof of Lemma \ref{Comparison in two arbitrary sequences of [0,1]} is again deferred to \cite[Appendix B]{RvdHMP}.}
\begin{proof}[Proof of Corollary \ref{general comparison}]
    We know that $\expec[r'_D]\leq \expec[r_D]$, since $\boldsymbol{r}\preccurlyeq_{\boldsymbol{p}} \boldsymbol{r}'$. Furthermore, if $\expec[r'_D] = \expec[r_D]$, then the result follows due to Theorem \ref{theorem-comparison-rs}. Thus, suppose $\expec[r'_D]<\expec[r_D]$, which means that there exists $\varepsilon>0$ such that $\expec[r_D] = \expec[r'_D]+\varepsilon$. By Lemma \ref{Comparison in two arbitrary sequences of [0,1]}, there exists an $\varepsilon$-sequence, say $\boldsymbol{\delta}$, such that $\boldsymbol{r}+\boldsymbol{\delta} \preccurlyeq_{\boldsymbol{p}} \boldsymbol{r}'$. Notice that both $\boldsymbol{r}+\boldsymbol{\delta}$ and $\boldsymbol{r}'$ are $\alpha$-sequences. Thus, from Theorem \ref{theorem-comparison-rs}, we have $\rho(\boldsymbol{r}') \leq \rho(\boldsymbol{r} +\boldsymbol{\delta})$. Also, by the definition of vertex removal according to an $\alpha$-sequence, it follows that $\rho(\boldsymbol{r}+\boldsymbol{\delta}) \leq \rho(\boldsymbol{r})$, because we are removing $\boldsymbol{\delta}$ proportion more vertices after $\boldsymbol{r}$-removal in $(\boldsymbol{r} +\boldsymbol{\delta})$-removal. Thus, we get the desired result.
\end{proof}
\fund
\noindent This work is supported in part by the Netherlands Organisation for Scientific Research (NWO) through the Gravitation NETWORKS grant no.\ 024.002.003.

{\footnotesize
\bibliographystyle{APT}
\bibliography{bib}
}
\appendix
\longversion{

\section{Stochastic ordering and giant in Configuration Model}
\label{app-A}
In this section, we prove Theorem \ref{Th: stochastic ordering in congiguration model}. For that we first show Lemma \ref{lemma: epsilon transformation step proof}, which essentially proves Theorem \ref{Th: stochastic ordering in congiguration model} for a small subcase. Lemma \ref{lemma: epsilon transformations and stochastic dominance} finishes the proof by showing that this was sufficient. First we define $\varepsilon$-transformations for probability measures:

\begin{defn}[$\varepsilon$-transformations for probability measures]
    Let $\boldsymbol{p} = (p_i)_{i\geq 0}$ be a probability mass function on the set of non-negative integers. For $\varepsilon\leq p_k$, define
    $$\boldsymbol{p}(\varepsilon) := (p_1, p_2, \ldots, p_{k-1}, p_k - \varepsilon, p_{k+1}, \ldots, p_{k+l-1}, p_{k+l}+\varepsilon, p_{k+l+1}, \ldots).$$
\end{defn}
Before moving to our first lemma, we recall that $\rho_{\sss\rm{CM}}(\boldsymbol{p})$ is the size of giant in the configuration model with limiting degree distribution having probability mass function  $\boldsymbol{p}=(p_i)_{i\geq1}$. Recall that $\rho_{\sss\rm{CM}}(\boldsymbol{p}) = 1-g_D(\eta)$, where $\eta$ satisfies $g'_D(\eta)= \expec[D]\eta$.
\begin{lem}\label{lemma: epsilon transformation step proof}
    For all $\varepsilon\leq p_k$, $\rho_{\sss\rm{CM}}(\boldsymbol{p}) \leq \rho_{\sss\rm{CM}}(\boldsymbol{p}(\varepsilon)).$
\end{lem}

\begin{proof}
    Let $g(\cdot)$ and $g_\varepsilon(\cdot)$ be the generating functions corresponding to $\boldsymbol{p}$ and $\boldsymbol{p}(\varepsilon)$, respectively. Let $\eta$ and $\eta_\varepsilon$ satisfy $g'(\eta) = \expec[D] \eta \text{ and } g_{\varepsilon}'(\eta_\varepsilon) = (\expec[D] +l\varepsilon) \eta_\varepsilon$.
    It is enough to show that $g_{\varepsilon}(\eta_\varepsilon)\leq g(\eta)$, since $\rho_{\sss\rm{CM}}(\boldsymbol{p}) = 1 - g(\eta)$. For this, we note that
\begin{equation}\label{eq: g varepsilon and g relationship}
    g_\varepsilon(s) = \sum_{i\geq 0}p_i(\varepsilon)s^i = g(s) - \varepsilon s^k(1-s^l).
\end{equation}
Thus, $g_\varepsilon(s) \leq g(s) \text{ for all } s\in [0,1].$ As a result, it is enough to show that $\eta_\varepsilon\leq \eta$. We differentiate \eqref{eq: g varepsilon and g relationship} with respect to $s$ to obtain
 \begin{align} \label{eq: g'varepsilon s}
     g'_\varepsilon(s) = g'(s)+(k+l)\varepsilon s^{k+l-1} - k\varepsilon s^{k-1}.
 \end{align}   
 Substitution of $\eta_\varepsilon$ in \eqref{eq: g'varepsilon s} gives
 \begin{equation}\label{eq: eta varepsilon}
     (\expec[D]+l\varepsilon)\eta_\varepsilon = g'(\eta_\varepsilon)+(k+l)\varepsilon \eta_\varepsilon^{k+l-1}-k\varepsilon \eta_\varepsilon^{k-1}.
 \end{equation}
Differentiating \eqref{eq: eta varepsilon} with respect to $\varepsilon$, we get 
\begin{equation}\label{eq: eta}
    \frac{\partial \eta_\varepsilon}{\partial \varepsilon} =  \frac{k\eta_\varepsilon^{k-1}-(k+l)\eta_\varepsilon^{k+l-1} +l\eta_\varepsilon}{g''(\eta_\varepsilon)-(\expec[D]+l\varepsilon) + \varepsilon (k+l)(k+l-1)\eta_\varepsilon^{k+l-2} - \varepsilon k(k-1)\eta_\varepsilon^{k-2})}.
\end{equation}
Differentiation of \eqref{eq: g varepsilon and g relationship} twice gives
\begin{equation}\label{eq: differentiation double g}
    g''_\varepsilon(s) = g''(s) - \varepsilon k(k-1) s^{k-2} + \varepsilon (k+l)(k+l-1)s^{k+l-2}.
\end{equation}
    Substituting \eqref{eq: differentiation double g} in \eqref{eq: eta} gives
\begin{equation}\label{eq: derivative of eta at varepsilon zero}
    \frac{\partial \eta_\varepsilon}{\partial \varepsilon} =  \frac{k\eta_{\varepsilon}^{k-1}-(k+l)\eta_{\varepsilon}^{k+l-1} +l\eta_\varepsilon}{g_\varepsilon''(\eta_\varepsilon)-(\expec[D]+l\varepsilon)}.
\end{equation}
     Let $h(x) = g_\varepsilon'(x) - (\expec[D]+l\varepsilon)x$. We know that $h(\eta_\varepsilon) = h(1) = 0$. Due to Rolle's theorem, there exists $\theta \in (\eta_\varepsilon, 1)$ so that $h'(\theta) = 0$, so that $g_\varepsilon''(\theta) = \expec[D]+l\varepsilon$. Since $g_\varepsilon''(\cdot)$ is an increasing function, $g_\varepsilon''(\eta_\varepsilon)\leq g_\varepsilon''(\theta) = \expec[D]+l\varepsilon$. Thus,
    \begin{align}\label{eq: denominator negative}
        g_\varepsilon''(\eta)- (\expec[D]+l\varepsilon) < 0.
    \end{align}
As a result, the denominator in \eqref{eq: derivative of eta at varepsilon zero} is negative. We claim that the numerator in \eqref{eq: derivative of eta at varepsilon zero} as a polynomial in $\eta$ is always positive.

Let $P(x) := ks^{k-2}-(k+l)s^{k+l-2}+l$. It is enough to show that $P(\cdot)$ is positive in $(0,1)$. 
Modifying the proof of Lemma \ref{positive polynomial}, we see that $\gamma$, the only maximiser of $P(\cdot)$ in $(0,1)$, is given by
    $$
    \gamma = \left(\frac{k(k-2)}{(k+l)(k+l-2)}\right)^{1/l}. 
    $$
Since there is no minimiser in the interval $(0, 1)$, the minimum occurs either in $0$ or in $1$. We have $P(0)= l$ and $P(1) = 0$, so that $P(x) > 0$ for all $x \in (0,1)$. Thus, the right-hand side in \eqref{eq: derivative of eta at varepsilon zero} is always negative. This completes the proof.
\end{proof}


\begin{lem}\label{lemma: epsilon transformations and stochastic dominance}
    Let $\boldsymbol{p}\preccurlyeq_{\rm{st}} \boldsymbol{q}$. Then $\boldsymbol{q}$ can be reached from $\boldsymbol{p}$ by a series of $\varepsilon$-transformations. 
\end{lem}

\begin{proof}
    Define the sets $I$ and $J$ that partition $\mathbb{N}$ as
\begin{equation}
    I := \{i\in \mathbb{N}\colon p_i< q_i\}\qquad \text{and} \qquad  J := \{i\in \mathbb{N}\colon p_i\geq q_i\}.
\end{equation}
Applying our $\varepsilon$-transformation on $(i, j)$ with $\varepsilon = \varepsilon_{i,j}$ for all $i\in I$ and $j\in J \cap [i+1, \infty)$ on the $\alpha$-sequence $\boldsymbol{p}$, we obtain the $\alpha$-sequence, say ${\boldsymbol{p}}_{\varepsilon}$, which satisfies
    $$
    {\boldsymbol{p}}_{\varepsilon} = \Bigg( \Big(p_i+ \sum_{j>i, j \in J}\varepsilon_{i,j} \Big)_{i\in I}, \Big(p_j- \sum_{i<j, i \in I}\varepsilon_{i,j} \Big)_{j\in J}\Bigg).
    $$
Then, ${\boldsymbol{p}}_\varepsilon = {\boldsymbol{q}}$, if $(\varepsilon_{i,j})_{(i,j)\in I \times J}$ satisfies  
\begin{align}\label{varepsilon conditions-gen}
    \sum_{j>i, j \in J} \varepsilon_{i,j} = q_i-p_i \text{ for all }i\in I \quad \text{and}\quad \sum_{i<j, i \in I} \varepsilon_{i,j} = p_j-q_j \text{ for all }j\in J.
\end{align}
Thus, it is enough to prove that the solution to \eqref{varepsilon conditions-gen} exists. We construct this solution recursively to complete this proof. Let $I = \{i_1, i_2, \ldots\}$, where $i_1 = \min I$ and $i_k$ is defined recursively as 
    $$i_k = \min I\setminus\{i_1, i_2, \ldots, i_{k-1}\},$$
and the minimum $\min I$ of a set of integers $I$ is its minimal element.
Let $J = \{j_1, j_2, \ldots\}$, where $j_1 = \min J$ and $j_k$ is defined recursively as 
    $$j_k = \min J\setminus\{j_1, j_2, \ldots, j_{k-1}\}.$$
Notice that $1\in I$ by stochastic domination. We next define some notation to complete our proof:
\begin{notation}[Sums of differences]\label{notation: table stochastic domination characterization} Let $s_0 = t_0 = 0$, and define $s_n = \sum_{k=1}^{n}(q_{i_k} - p_{i_k})$ and $t_n= \sum_{k=1}^{n}(p_{j_k} -q_{j_k}).$
\end{notation}
We define $\varepsilon_{i,j}$ recursively to satisfy \eqref{varepsilon conditions-gen} in the following way: We define $n_0 = l_0 = 0$. We first prove that there exists $0\leq n_1 <j_1-1$ such that
\begin{align}
    s_{n_1} < t_{1}\qquad \text{and} \qquad s_{n_1+1} \geq t_{1},
\end{align}
and there exists $1\leq l_1$, such that 
\begin{align}
    s_{n_1} < t_{l_1}\qquad \text{and} \qquad  s_{n_1} \geq t_{l_1+1}.
\end{align}
Indeed, by stochastic domination,
    $$
    s_{j_1 - 1}-t_1 = \sum_{k=1}^{j_1-1}(q_k-p_k)  - (p_{j_1}-q_{j_1}) = \sum_{k=1}^{j_1}q_k-\sum_{k=1}^{j_1}p_k > 0.
    $$
We have $n_1+1 = \min\{i\in \mathbb{N}\colon s_i \geq t_1\}$, and $s_{j_1-1}\geq t_1$. Thus, $n_1<j_1 -1$. Since $j_1=\min J$, this means that $\{1, 2, 3,\ldots, j_1-1\}\subset I$. Thus, $i_{n_1} = n_1 < j_1-1$ and $i_{n_1+1} = n_1+1< j_1$. Having obtained $n_1$ and $l_1$, we define $\varepsilon_{i, j}$ for $1\leq i\leq i_{n_1+1}$ and $j_1\leq j\leq j_{l_1+1}$ by substituting $i'=0$ in Table \ref{epsilon values table-gen-2}. This defines $\varepsilon_{i, j}$ for $1\leq i\leq i_{n_1+1}$ and $j_1\leq j\leq j_{l_1+1}$. Further, we set $\varepsilon_{i, j} = 0$ for all $i> n_1+1$ and $j< j_{l_1}$. 

We next iterate this argument, and prove that, similarly, there exists $n_2 > n_1$ satisfying
\begin{align}
    s_{n_2} -s_{n_1} < t_{l_1+1}-t_{l_1}\quad \text{and} \quad s_{n_2+1}-s_{n_1} \geq t_{l_1+1}-t_{l_1},
\end{align}
and there exists $l_2 > l_1$, such that 
\begin{align}
     s_{n_2} -s_{n_1} < t_{l_2}-t_{l_1}\quad \text{and} \quad s_{n_2+1}-s_{n_1} \geq t_{l_2+1}-t_{l_1}.
\end{align}
Similarly as before, by stochastic domination, $i_{n_2+1}<j_{l_1+1}$. Having gotten $n_2$ and $l_2$, we define $\varepsilon_{i, j}$ for $i_{n_1+1} \leq i \leq i_{n_2+1}$ and $j_{l_1+1} \leq j\leq j_{l_2+1}$ by substituting $i'=1$ in Table \ref{epsilon values table-gen-2}.
\begin{table}[ht!]\caption{Values of $\varepsilon_{i,j}$}\label{epsilon values table-gen-2}
\centering
\small\begin{tabular}{|c|c|c|c|c|c|c|c|} 
 \hline
 \diagbox[width=3em]{$i$}{$j$} & $j_{{l_{i'}}+1}$ & $j_{{l_i'}+2}$ & $\cdots$ & $j_{l_{i'+1}}$& $j_{l_{i'+1}+1}$ & . & . \\ [0.5ex] 
 \hline
 $i_{n_{i'}+1}$ & $ s_{n_{i'} +1}-t_{l_{i'}}$ & 0 & $\cdots$ & 0 & 0 & . & .\\ 
 $i_{n_{i'}+2}$ & $ q_{n_{i'} +2}-p_{n_{i'} +2}$ & 0 & $\cdots$ & 0 & 0 & . & .\\ 
 $i_{n_{i'}+3}$ & $ q_{n_{i'} +3}-p_{n_{i'} +3}$ & 0 & $\cdots$ & 0 & 0 &. & .\\
 . & . & . & $\cdots$ & . & . & . & .\\
 . & . & . & $\cdots$ & . & . & . & .\\
 $i_{n_{i'+1}}$ & $q_{n_{i'+1}}-p_{n_{i'+1}}$ & 0 & $\cdots$ & . & . & . & .\\ 
 $i_{n_{i'+1}+1}$ & $t_{l_{i'}+1}-s_{n_{i'+1}}$ & $t_{l_{i'}+2}-t_{l_{i'}+1}$ & $\cdots$ & $t_{l_{i'+1}}-t_{l_{i'+1}-1}$ & $s_{n_{i'+1}+1} - t_{l_{i'+1}}$  & 0 & . \\[1ex] 
 \hline
 \end{tabular}
\end{table}

We next distinguish different cases, depending on whether $I$ and $J$ are finite or infinite.

Suppose first that both $I$ and $J$ are infinite sets. In that case the recursion never stops and we keep on defining $(n_s, j_s)$, and, in turn, $\varepsilon_{i,j}$ for all values of $(i,j)$, hence getting the solution for \eqref{varepsilon conditions-gen}. 

Suppose next that one of $I$ or $J$ is finite, then due to stochastic domination, the set $I$ will have to be the one to be finite. Suppose $\max I = k$. Then there will exist $m\in \mathbb{N}$ such that $i_{n_m+1} = k$. Then we define $\varepsilon_{i, j}$ in Table \ref{epsilon values table-gen-3}. To complete the proof, we need to show
    \begin{equation}
    \label{need to show}
    s_{n_m +1} = \lim_{n\to \infty}t_n.
    \end{equation}
This is because the sum of the $k$th row is $\lim_{n\to\infty}t_n - s_{n_m}$. If \eqref{need to show} is true, then the $k$th row sum becomes $q_k-p_k$, which is as desired in \eqref{varepsilon conditions-gen}.
\begin{table}[ht!]\caption{Values of $\varepsilon_{i,j}$}\label{epsilon values table-gen-3}
\centering
 \small\begin{tabular}{|c| c| c| c|c | c|c|} 
 \hline
 \diagbox[width=3em]{$i$}{$j$} & $j_{{l_{m-1}}+1}$ & $j_{{l_{m-1}}+2}$  & $\cdots$ & $j_n$ & $j_{n+1}$ & $j_{n+2}$ \\ [0.5ex] 
 \hline
 $i_{n_{m-1}+2}$ & $q_{n_1 +2}-p_{n_1 +2}$ & $0$ & $\cdots$ & $0$ & $0$ & $\cdots$\\ 
 $i_{n_{m-1}+3}$ & $q_{n_1 +3}-p_{n_1 +3}$ & $0$ & $\cdots$ & $0$ & $0$ & $\cdots$\\
 . & . & . & $\cdots$ & . & . & .\\
 . & . & . & $\cdots$ & . & . & .\\
 $i_{n_m}$ & $q_{n_2}-p_{n_2}$ & $0$ & $\cdots$ & $0$ & $0$ & $\cdots$\\ 
 $k= i_{n_m+1}$ & $t_{l_{m-1}+1}-s_{n_m}$ &  $t_{l_{m-1}+2}-t_{l_{m-1}+1}$ & $\cdots$ & $t_n -t_{n-1}$ & $t_{n+1} -t_{n}$ & $\cdots$ \\[1ex] 
 \hline
 \end{tabular}
\end{table}
Now we show \eqref{need to show}. Since $I$ and $J$ partition the natural numbers, we have $I\cup J = \mathbb{N}$ and thus
\begin{align*}
    &\sum_{i \in I \cup J}p_i = \sum_{j\in I \cup J}q_i = 1  \text{, so that } \sum_{i\in I}p_i + \sum_{i\in J}p_i = \sum_{i\in I}q_i+ \sum_{j\in J}q_j.
\end{align*}
In turn this implies that
\begin{equation*}
    \sum_{i\in I}(p_i -q_i) = \sum_{j\in J}(q_j -p_j).
\end{equation*}
Notice that, by definition, $\lim_{n\to \infty}t_n = \sum_{i\in I}(p_i -q_i)$ and $ s_{n_m +1} =  \sum_{j\in J}(q_j -p_j)$. This finishes the proof.
\end{proof}
\begin{proof}[Proof of Theorem \ref{Th: stochastic ordering in congiguration model}]
    Lemma \ref{lemma: epsilon transformation step proof} and \ref{lemma: epsilon transformations and stochastic dominance} complete the proof.
\end{proof}

\section{Characterization of stochastic ordering in $\alpha$-sequences }
\label{app-B}
In this section, we prove Lemmas \ref{characterization of alpha sequences order} and \ref{Comparison in two arbitrary sequences of [0,1]}.  By definition of stochastic domination, $\boldsymbol{r} \preccurlyeq_{\boldsymbol{p}} \boldsymbol{r}'\iff \boldsymbol{p}\boldsymbol{r} \preccurlyeq_{st} \boldsymbol{p}\boldsymbol{r}'.$ Thus, Lemma \ref{characterization of alpha sequences order} follows from Lemma \ref{lemma: epsilon transformations and stochastic dominance}. Therefore, we only prove Lemma \ref{Comparison in two arbitrary sequences of [0,1]}:

\begin{proof}[Proof of Lemma \ref{Comparison in two arbitrary sequences of [0,1]}]
    We define our $\varepsilon$-sequence $(\delta_i)_{i\geq 1}$ recursively. First, we define
\begin{equation}\label{k_1}
    k_1  := \left.
                        \begin{cases}
                         \min \left\{n\in \mathbb{N}:  r'_n < r_n \right\}  & \text{if the minimum is attained},  \\
                        \infty    &  \text{otherwise}.
                       \end{cases}
                       \right.
\end{equation} 
We know that $k_1 >1$, due to $\boldsymbol{r}\preccurlyeq_{\boldsymbol{p}} \boldsymbol{r}'$. Define $\delta_i = r'_i - r_i$ for each $i<k_1$. If $k_1 = \infty$, we define $\delta_i=r'_i - r_i$ for all $i\in \mathbb{N}$, so that, by definition of $k_1$, $\delta_i=r'_i - r_i\geq 0$, and thus it satisfies the required condition.
Suppose instead that $k_1<\infty$. Then we define
    $$
    l_1 = \min\left\{n >k_1: \sum_{i=k_1}^{n}p_i(r'_i - r_i)>0\right\}.
    $$
Since $\boldsymbol{r}\preccurlyeq_{\boldsymbol{p}} \boldsymbol{r}'$ and $k_1<\infty$, by definition we have $\sum_{i=k_1}^{\infty}p_i(r'_i - r_i)>0$. As a result, there exists $m$ such that $\sum_{i=k_1}^{m}p_i(r'_i - r_i)>0$, which means that the set whose minimum we want to compute is non-empty. By the well ordering principle for natural numbers, we can conclude that the minimum is attained and thus, $l_1 < \infty$. Having obtained that $l_1 <\infty$, we define $\delta_i$ for all $1 \leq i \leq l_1$ as
\begin{equation}\label{delta}
    \delta_i  := \left.
                        \begin{cases}
                         r'_i - r_i, & \text{for }i<k_1,  \\
                        0,   &  \text{for }k_1\leq i< l_1, \\
                        \frac{1}{p_{l_1}}\sum_{i=k_1}^{l_1}(r'_i-r_i)p_i, & \text{for }i = l_1.
                       \end{cases}
                       \right.
\end{equation}
Next we define $k_2$, similarly to $k_1$, as
\begin{equation}\label{k_2}
    k_2  := \left.
                        \begin{cases}
                         \min\left\{n >l_1: r'_n < r_n\right\}  & \text{if the minimum is attained},  \\
                        \infty    &  \text{otherwise}.
                       \end{cases}
            \right.
\end{equation}
 If $k_2 = \infty$, then we define $\delta_i=r'_i - r_i$ for all $i>l_1$ and it again satisfies the required condition.
    Suppose instead that $k_2<\infty$, then we define $l_2$ similarly to $l_1$ as
    $$l_2 = \min\left\{n >k_2: \sum_{i=k_2}^{n}p_i(r'_i - r_i)>0\right\}.$$
    Since $\boldsymbol{r}\preccurlyeq_{\boldsymbol{p}} \boldsymbol{r}'$, it can be shown that the above minimum is always attained and $l_2<\infty$. We define, $\delta_i$ for all $k_2 \leq i \leq l_2$ as
\begin{equation}\label{delta1}
    \delta_i  := \left.
                        \begin{cases}
                         r'_i - r_i & \text{for }l_1\leq i<k_2,  \\
                         0  &  \text{for }k_2\leq i< l_2, \\
                        \frac{1}{p_{l_1}}\sum_{i=k_1}^{l_1}(r'_i-r_i)p_i & \text{for }i = l_1.
                       \end{cases}
                \right.
\end{equation}
We keep on defining $k_i$ and $l_i$ similarly, unless $k_i=\infty$ for some $i\geq 1$. In this way, we get the desired $\varepsilon$-sequence iteratively.
\end{proof}

}